\documentclass[11pt]{amsart}
\RequirePackage{mathtools}
\RequirePackage{amsopn}
\RequirePackage{amsfonts}
\RequirePackage{paralist}
\RequirePackage{amssymb}
\RequirePackage{amsthm}
\RequirePackage{mathrsfs}
\usepackage[alphabetic]{amsrefs}
\renewcommand\MR[1]{\relax} 
\usepackage{tikz}
\usetikzlibrary{cd,decorations.pathmorphing}
\usepackage{mathrsfs}

\newtheorem{thm}{Theorem}[section]
\numberwithin{equation}{section}

\newtheorem{lemma}[thm]{Lemma}
\newtheorem{prop}[thm]{Proposition}

\theoremstyle{definition}

\theoremstyle{remark}
\newtheorem{remark}[thm]{Remark}
\newtheorem{example}[thm]{Example}

\newtheorem{mycomment}[thm]{Comment}
{\end{mycomment}\endgroup}
\def\mathcs{C^{*}}
\newcommand{\cs}{\ensuremath{\mathcs}}

\DeclareMathSymbol{\rtimes}{\mathbin}{AMSb}{"6F}
\newcommand{\ib}{im\-prim\-i\-tiv\-ity bi\-mod\-u\-le}

\newcommand{\sme}{\,\mathord{\mathop{\text{--}}\nolimits_{\relax}}\,}

\def\R{\mathbf{R}}
\def\C{\mathbf{C}}
\def\NN{\mathbf{N}}
\def\T{\mathbf{T}}
\def\Z{\mathbf{Z}}

\def\Q{\mathbf{Q}}

\DeclareMathOperator*{\supp}{supp}

\DeclareMathOperator{\id}{id}

\def\set#1{\{\,#1\,\}}
\newcommand\sset[1]{\{#1\}}

\def\restr#1{|_{{#1}}}
\makeatletter
\def\labelenumi{\textnormal{(\@alph\c@enumi)}}
\def\theenumi{\@alph \c@enumi}
\def\labelenumii{\textnormal{(\@roman\c@enumii)}}
\def\theenumii{\@roman \c@enumii}
\newcount\charno
\def\alphapart#1{\charno=96
\advance\charno by#1\char\charno}

\makeatother

\def\<{\langle}
\def\>{\rangle}
\let\ipscriptstyle=\scriptscriptstyle
\def\lipsqueeze{{\mskip -3.0mu}}
\def\ripsqueeze{{\mskip -3.0mu}}
\def\ipcomma{\nobreak\mathrel{,}\nobreak}
\newbox\ipstrutbox
\setbox\ipstrutbox=\hbox{\vrule height8.5pt
width 0pt}
\def\ipstrut{\copy\ipstrutbox}
\def\lip#1<#2,#3>{\mathopen{\relax_{\ipstrut\ipscriptstyle{
#1}}\lipsqueeze
\langle} #2\ipcomma #3 \rangle}
\def\blip#1<#2,#3>{\mathopen{\relax_{\ipstrut
\ipscriptstyle{ #1}}\lipsqueeze\bigl\langle} #2\ipcomma #3 \bigr\rangle}
\def\rip#1<#2,#3>{\langle #2\ipcomma #3
\rangle_{\ripsqueeze\ipstrut\ipscriptstyle{#1}}}
\def\brip#1<#2,#3>{\bigl\langle #2\ipcomma #3
\bigr\rangle_{\ripsqueeze\ipstrut\ipscriptstyle{#1}}}
\def\angsqueeze{\mskip -6mu}
\def\smangsqueeze{\mskip -3.7mu}
\def\trip#1<#2,#3>{\langle\smangsqueeze\langle #2\ipcomma #3
\rangle\smangsqueeze\rangle_{\ripsqueeze\ipstrut\ipscriptstyle{#1}}}
\def\btrip#1<#2,#3>{\bigl\langle\angsqueeze\bigl\langle #2\ipcomma
#3
\bigr\rangle
\angsqueeze\bigr\rangle_{\ripsqueeze\ipstrut\ipscriptstyle{#1}}}
\def\tlip#1<#2,#3>{\mathopen{\relax_{\ipstrut\ipscriptstyle{
#1}}\lipsqueeze \langle\smangsqueeze\langle} #2\ipcomma #3
\rangle\smangsqueeze\rangle}
\def\btlip#1<#2,#3>{\mathopen{\relax_{\ipstrut\ipscriptstyle{
#1}}\lipsqueeze
\bigl\langle\angsqueeze\bigl\langle} #2\ipcomma #3
\bigr\rangle\angsqueeze\bigr\rangle}

\def\ip(#1|#2){(#1\mid #2)}
\def\bip(#1|#2){\bigl(#1 \mid #2\bigr)}
\def\Bip(#1|#2){\Bigl( #1 \bigm| #2 \Bigr)}
\expandafter\ifx\csname BibSpec\endcsname\relax\else
\BibSpec{collection.article}{%
    +{}  {\PrintAuthors}                {author}
    +{,} { \textit}                     {title}
    +{.} { }                            {part}
    +{:} { \textit}                     {subtitle}
    +{,} { \PrintContributions}         {contribution}
    +{,} { \PrintConference}            {conference}
    +{}  {\PrintBook}                   {book}
    +{,} { }                            {booktitle}
    +{,} { }                            {series}
    +{,} { \voltext}                    {volume}
    +{,} { }                            {publisher}
    +{,} { }                            {organization}
    +{,} { }                            {address}
    +{,} { \PrintDateB}                 {date}
    +{,} { pp.~}                        {pages}
    +{,} { }                            {status}
    +{,} { \PrintDOI}                   {doi}
    +{,} { available at \eprint}        {eprint}
    +{}  { \parenthesize}               {language}
    +{}  { \PrintTranslation}           {translation}
    +{;} { \PrintReprint}               {reprint}
    +{.} { }                            {note}
    +{.} {}                             {transition}
}
\BibSpec{article}{%
    +{}  {\PrintAuthors}                {author}
    +{,} { \textit}                     {title}
    +{.} { }                            {part}
    +{:} { \textit}                     {subtitle}
    +{,} { \PrintContributions}         {contribution}
    +{.} { \PrintPartials}              {partial}
    +{,} { }                            {journal}
    +{}  { \textbf}                     {volume}
    +{}  { \PrintDatePV}                {date}
    +{,} { \eprintpages}                {pages}
    +{,} { }                            {status}
    +{,} { \PrintDOI}                   {doi}
    +{,} { available at \eprint}        {eprint}
    +{}  { \parenthesize}               {language}
    +{}  { \PrintTranslation}           {translation}
    +{;} { \PrintReprint}               {reprint}
    +{.} { }                            {note}
    +{.} {}                             {transition}
}
\BibSpec{book}{%
    +{}  {\PrintPrimary}                {transition}
    +{,} { \textit}                     {title}
    +{.} { }                            {part}
    +{:} { \textit}                     {subtitle}
    +{,} { \PrintEdition}               {edition}
    +{}  { \PrintEditorsB}              {editor}
    +{,} { \PrintTranslatorsC}          {translator}
    +{,} { \PrintContributions}         {contribution}
    +{,} { }                            {series}
    +{,} { \voltext}                    {volume}
    +{,} { }                            {publisher}
    +{,} { }                            {organization}
    +{,} { }                            {address}
    +{,} { pp.~}                        {pages}
    +{,} { \PrintDateB}                 {date}
    +{,} { }                            {status}
    +{}  { \parenthesize}               {language}
    +{}  { \PrintTranslation}           {translation}
    +{;} { \PrintReprint}               {reprint}
    +{.} { }                            {note}
    +{.} {}                             {transition}
}
\fi
\evensidemargin=\oddsidemargin
\mathtoolsset{showonlyrefs} 

\newcommand\A{\mathcal A}

\newcommand\hA{\hat \A}
\newcommand{\B}{\mathcal B}
\newcommand{\D}{\mathcal D}
\newcommand{\G}{\mathcal G}
\newcommand\Ll{\mathcal L}

\newcommand\Sigmaw{\widehat\Sigma}
\newcommand\ccsas{C_{c}^{\Sigmaw}(\hat\A\rtimes\Sigma)}
\newcommand\cisas{C_{\infty}^{\Sigmaw}(\hat\A\rtimes\Sigma)}
\newcommand\half{\frac12}
\newcommand\go{\Sigma^{(0)}}
\newcommand\goo{\G^{(0)}}
\newcommand\rgo{G^{(0)}}
\newcommand\p{\pi}
\newcommand\hp{\hat\p}

\newcommand\cshacgsw{\cs(\hA\rtimes\G;\Sigmaw)}
\newcommand\csrhacgsw{\cs_r(\hA\rtimes\G;\Sigmaw)}
\newcommand\cc{C_{c}}

\newcommand\HH{\mathscr{H}}

\renewcommand\H{\mathcal{H}}

\newcommand\ccs{\cc(\Sigma)}
\newcommand\hacs{\hA\rtimes\Sigma}

\newcommand\hacg{\hA\rtimes\G}

\newcommand\Sec{\Gamma}

\newcommand\F{\mathcal{F}}

\newcommand\bbeta{m}

\newcommand\auto{\vartheta}
\newcommand\hauto{\hat\auto}

\newcommand\renj{j}
\newcommand\ilt{inductive-limit topology}
\DeclareMathOperator{\Iso}{Iso}
\DeclareMathOperator{\End}{End}

\newcommand\Rep{\underline{L}}
\newcommand\Deltab{\underline{\Delta}}
\newcommand\bpi{\bar\pi}
\newcommand\gtw{\kappa}
\newcommand\hgtw{\hat\kappa}

\newcommand\halpha{\underline{\alpha}}

\newcommand\N{\mathcal{N}}

\newcommand\gux{G\backslash X}
\newcommand\ex{\E^{X}}
\newcommand\px{p^{X}}
\newcommand\usc{upper-semicontinuous}
\newcommand\E{\mathscr{E}}
\newcommand\BB{\mathscr{B}}
\newcommand\ccgsgtw{\cc(\G,\Sigma,\E,\gtw)}
\newcommand\csgsgtw{\cs(\G,\Sigma,\E,\gtw)}
\newcommand\csrgsgtw{\cs_r(\G, \Sigma, \E, \gtw)}
\newcommand\ccgshgtw{\cc(\G,\Sigma,\hat\E,\hat\gtw)}
\newcommand\csgshgtw{\cs(\G,\Sigma,\hat\E,\hat\gtw)}
\newcommand\csrgshgtw{\cs_r(\G,\Sigma,\hat\E,\hat\gtw)}
\newcommand{\Stab}{\operatorname{Stab}}

\begin{document}
\begin{abstract}
  Given a normal subgroup bundle $\A$ of the isotropy bundle of a
  groupoid $\Sigma$, we obtain a twisted action of the quotient
  groupoid $\Sigma/\A$ on the bundle of group $C^*$-algebras determined
  by $\A$ whose twisted crossed product recovers the groupoid
  $C^*$-algebra $C^*(\Sigma)$. Restricting to the case where $\A$ is
  abelian, we describe $C^*(\Sigma)$ as the $C^*$-algebra associated
  to a $\T$-groupoid over the tranformation groupoid obtained from the
  canonical action of $\Sigma/\A$ on the Pontryagin dual space of
  $\A$. We give some illustrative examples of this result.
\end{abstract}

\title{\cs-Algebras of Extensions of Groupoids by Group Bundles}
\date{6 November 2020}

\author[M. Ionescu]{Marius Ionescu}
\address{Department of Mathematics\\  United States Naval Academy\\
  Annapolis, MD 21402 USA}
\email{ionescu@usna.edu}

\author[A. Kumjian]{Alex Kumjian}
\address{Department of Mathematics \\ University of Nevada\\ Reno NV
  89557 USA}
\email{alex@unr.edu}

\author{Jean N. Renault}
\address{Institut Denis Poisson (UMR 7013)\\
  Universit\'e d'Orl\'eans et CNRS \\ 45067
  Orl\'eans Cedex 2, FRANCE}
\email{jean.renault@univ-orleans.fr}

\author[A. Sims]{Aidan Sims}
\address{School of Mathematics and Applied Statistics\\ University of
  Wollongong\\ NSW 2522, Australia}
\email{asims@uow.edu.au}

\author[D. P. Williams]{Dana P. Williams}
\address{Department of Mathematics\\ Dartmouth College \\ Hanover, NH
  03755-3551 USA}
\email{dana.williams@Dartmouth.edu}

\thanks{This research was supported by Australian Research Council
  grant DP180100595. This work was also partially supported by Simons
  Foundation Collaboration grants \#209277 (MI), \#353626 (AK) and
  \#507798 (DPW), and by Grant N0017318WR00251 from the Office of
  Naval Research and by the Naval Research Laboratory. This work was
  also facilitated by visits of the
  second and fifth author to the University of Wollongong as well as a
visit by the second author to the Universit\'e d'Orl\'eans.  We thank
our respective hosts and their institutions for their hospitality and
support.
We thank the anonymous referee of an earlier version of this
  paper for very thorough and detailed report that allowed us to
  improve the exposition.}

\maketitle

\section*{Introduction}

The objective of this paper is to develop tools for analyzing
$C^*$-algebras of very general groupoids via a version of the Mackey
machine, and also to construct Cartan subalgebras in $C^*$-algebras of
large classes of groupoids that are not necessarily
\'etale. Specifically, one of the fundamental components of the modern
Mackey Subgroup Analysis for group representations comes from Green's
\cite{gre:am78}*{Proposition~1}: if $H$ is a normal subgroup of a
locally compact group $G$, then $\cs(G)$ is a twisted crossed product
$\cs(G,\cs(H),\gtw)$ of $\cs(H)$ by $G$. The idea is then to describe
the representation theory and the structure of $\cs(G)$ in terms of
representations of $H$ and the action of $G$ on the spectrum of
$\cs(H)$ induced by conjugation. Variants of this approach are
sometimes called ``Mackey's Little Group Method''.

Our main theorem generalizes Green's theorem to the situation of a
locally compact Hausdorff groupoid $\Sigma$ and a normal subgroup
bundle $\A$ of the isotropy groupoid of $\Sigma$. Our result says that
the full $C^*$-algebra of $\Sigma$ coincides with a twisted crossed
product of $C^*(\A)$ by $\Sigma$. If $\A$ is amenable, this descends
to an isomorphism of reduced $C^*$-algebras. We pay particular
attention to the situation when $\A$ is abelian, and hence has a
Pontryagin dual space $\hA$. The quotient $\G=\Sigma/\A$ is a
topological groupoid, and we use our main theorem to prove that
$\cs(\Sigma)$ is isomorphic to the restricted \cs-algebra
$\cs(\hA\rtimes\G;\Sigmaw)$ of a $\T$-groupoid $\Sigmaw$ over a
transformation groupoid $\hA\rtimes\G$.  As the theory of
$\T$-groupoids---also called twists---is well studied, this provides
powerful tools for studying $\cs(\Sigma)$. This was the approach for
the main result in \cite{mrw:tams96}, and in fact, our isomorphism
result is a significant generalization of, and is motivated by,
\cite{mrw:tams96}*{Proposition~4.5}. As an illustrative example, we
discuss the special case of a locally compact Hausdorff group $\Sigma$
with a closed normal abelian subgroup $\A$; even in this special case,
our results have something new to say.

Accordingly, in Section~\ref{sec:extens-abel-group}, we consider a
second countable locally compact Hausdorff groupoid and subgroup
bundle $\A$ of the isotropy groupoid of $\Sigma$ whose unit space
coincides with that of $\Sigma$.  Then $\A$ acts on the left and right
of $\Sigma$ and we say that $\A$ is normal if the orbits $\sigma\A$
and $\A\sigma$ coincide for all $\sigma\in\Sigma$.  If $\A$ is normal,
then the quotient $\G=\Sigma/\A$ is a groupoid and we can view
$\Sigma$ as an extension of $\A$ by $\G$.  Assuming that both $\Sigma$
and $\A$ have Haar systems, $\Sigma$ acts by automorphisms on
$\cs(\A)$. We show that there is a Green--Renault twisting map $\gtw$
for this action.  This allows us to form a twisted crossed product of
which Green's twisted crossed products are a special case.  When $\A$
is an abelian group bundle these twisted crossed products are the same
as those in \cites{ren:jot91,ren:jot87}.  Our first main result
(Theorem~\ref{thm-hm}) extends Green's result by proving that
$\cs(\Sigma)$ is isomorphic to the twisted crossed product; it also
establishes a similar result for the reduced norm when $\A$ is
amenable. There is overlap with a result in \cite{bm:xx16}; we discuss
this in detail in Remark~\ref{rem-bm}.

In Section~\ref{sec:abelian-case} we restrict to the case that $\A$ is
an abelian group bundle.  Then the Gelfand dual, $\hA$, of $\cs(\A)$
is a locally compact right $\G$-space and we can form the
transformation groupoid $\hA\rtimes\G$.  We can build a natural
$\T$-groupoid, $\Sigmaw$, over $\hA\rtimes \G$.  Our second main
result (\ref{thm-main-mrw96}) is that $\cs(\Sigma)$ is isomorphic to
the restricted groupoid \cs-algebra $\cs(\hA\rtimes \G;\Sigmaw)$ of
this $\T$-groupoid with a similar assertion for the reduced norms. It
follows from \cite{barli:am17} that $C^*(\Sigma)$ belongs to the UCT
class if and only if it is nuclear.

In Section~\ref{sec:more examples}, we give a number of applications
and examples of our results.  We first consider in
Section~\ref{sec:clos-norm-abel} the situation of a closed normal
abelian subgroup $H$ of a locally compact group $G$. The quotient
$G/H$ acts on $H$ by conjugation, and hence also on the Pontryagin
dual $\widehat{H}$ viewed as a topological space. Our construction
yields a $\T$-groupoid $\Sigmaw$ over $\widehat{H} \rtimes (G/H)$. Our
main theorem identifies the group $C^*$-algebra $C^{*}(G)$ with the
$C^*$-algebra $C^*(\widehat{H} \rtimes (G/H); \Sigmaw)$ of this
$\T$-groupoid, and similarly at the level of reduced $C^*$-algebras
(cf., \cite{zel:jmpa68}).  As specific examples of this result we
demonstrate how to recover two standard descriptions of the
$C^*$-algebra of the integer Heisenberg group (see
Example~\ref{ex:Heisenberg}). Our second class of examples is that of
extensions of effective \'etale groupoids $\G$ by bundles $\A$ of
abelian groups. In this situation, the semidirect product
$\hA \rtimes \G$ is also \'etale, and we prove in Theorem~\ref{cartan}
that $C_0(\hA)$ embeds as a Cartan subalgebra of $C^*_r(\Sigma)$. This
generalizes a result in \cite{bnrsw:ieot16}. We conclude
Section~\ref{sec:more examples} with a few examples including
applications to higher-rank Deaconu--Renault groupoids and, in
particular, groupoids associated to row-finite higher rank graphs with
no sources.

This paper is a revised and shortened version of \cite{ikrsw:xx20c}.


\section{Preliminaries}\label{sec:prelims}

Our results require that we work with Green--Renault twisted crossed
products of \cs-algebras by locally compact Hausdorff groupoids, and
in particular, the theory of $\T$-groupoids and their restricted
groupoid \cs-algebras.  In order to take a unified approach, we use
the theory of Fell bundles over groupoids and their associated
\cs-algebras.  For convenience, we give a brief overview of the
required background with selected references.

\subsection{Open maps}
\label{sec:open-maps}

Open maps are ubiquitous when working with groupoids and bundles.  An
important tool we refer to frequently is known as \emph{Fell's
  Criterion}. For convenience we recall the formal statement from
\cite{fd:representations1}*{Proposition~II.13.2}.

\begin{lemma}[Fell's Criterion]
  \label{lem-fell-criterion}
  A surjection $f : X \to Y$ is an open map if and only if for every
  $x \in X$ and every net $\set{y_\alpha}$ converging to $f(x)$ in $Y$
  there is a subnet $\set{y_{\alpha_j}}$ and a net
  $\set{x_{\alpha_j}}$ in $Y$ such that
  $f(x_{\alpha_j}) = y_{\alpha_j}$ for all $j$, and
  $x_{\alpha_j} \to x$.
\end{lemma}

\subsection{Upper-semicontinuous Banach bundles}
\label{sec:upper-semic-banach}

A (upper-semicontinuous) Banach bundle over a space $X$ is a
topological space $\E$ together with a continuous open surjection
$p:\E\to X$ and complex Banach space structures on each fibre
$E(x):=p^{-1}(x)$ such that
\begin{compactenum}
\item $a\mapsto \|a\|$ is upper-semicontinuous from $\E$ to $\R^{+}$,
\item $(a,b)\mapsto a+b$ is continuous from
  $\E*\E=\set{(a,b)\in \E\times\E:p(a)=p(b)}$ to $\E$,
\item $(\lambda,a)\mapsto \lambda a$ is continuous from $\C\times\E$
  to $\E$, and
\item If $\set{a_{i}}$ is a net in $\E$ such that $p(a_{i})\to x$ and
  $\|a_{i}\|\to 0$, then $a_{i}\to 0_{x}$ where $0_{x}$ is the zero
  element in $E(x)$.
\end{compactenum}
If in addition each fibre $E(x)$ is a \cs-algebra and
$(a,b)\mapsto ab$ and $a\mapsto a^{*}$ are continuous on $\E*\E$ and
$\E$, respectively, then we call $\E$ a (upper-semicontinuous)
\cs-bundle.  If axiom~(a) is replaced by ``$a\mapsto \|a\|$ is
continuous'', then we call $\E$ either a continuous Banach bundle or a
continuous \cs-bundle.  We normally drop the adjective
``upper-semicontinuous'' and add ``continuous'' only when we
specialize to that case.

An excellent reference for continuous Banach bundles is \S\S13--14 of
\cite{fd:representations1}*{Chap. II}.  For more details in the
general case see \cite{muhwil:dm08}*{Appendix~A} and, for the
\cs-bundle case, \cite{wil:crossed}*{Appendix~C}. If $p:\E\to X$ is a
Banach bundle, we will write $\Sec_{0}(X;\E)$ for the continuous
sections of $\E$ which vanish at infinity.  Lazar's
\cite{laz:jmaa18}*{Theorem~2.9} shows that every Banach bundle
$p:\E\to X$ is guaranteed to have sufficiently many sections in the
sense that given $a\in E(x)$ there is a $f\in \Sec_{0}(X;\E)$ such
that $f(x)=a$.  Futhermore, $\Sec_{0}(X;\E)$ is a Banach space with
respect to the supremum norm. A primary motivation for working with
upper-semicontinuous bundles rather than restricting to continuous
ones is that a $C_{0}(X)$-algebra $A$ is always $C_{0}(X)$-isomorphic
to $\Sec_{0}(X,\mathscr A)$ for an appropriate upper-semicontinuous
\cs-bundle $\mathscr A$ \cite{wil:crossed}*{Theorem~C.26}.

In our constructions below, we need to work with Banach bundles that
arise as quotients of Banach bundles by isometric groupoid actions as
in \cite{kmrw:ajm98}*{\S2} (where these constructions were developed
for continuous bundles).  Specifically, we let $G$ be a second
countable locally compact Hausdorff groupoid with a Haar system
$\lambda$.  We suppose that $p:\E\to \rgo$ is an Banach bundle
admitting a continuous $G$-action $\gamma\cdot a = \alpha_{\gamma}(a)$
where $\alpha_{\gamma}:E(s(\gamma))\to E(r(\gamma))$ is an
\emph{isometric} Banach space isomorphism of the fibres of $\E$.

Now let $X$ be a \emph{free} and proper left $G$-space.  We can form
the pullback $X*\E=\set{(x,a)\in X\times \E:r(x)=p(a)}$.  Then $X*\E$
is an Banach bundle over $X$ which is a continuous Banach bundle if
$\E$ is.  We get a left $G$-action on $X*\E$ given by
\begin{equation}
  \label{eq:51}
  \gamma\cdot (x,a)= (\gamma\cdot x,\alpha_{\gamma}(a)).
\end{equation}
We will write $\ex$ for the orbit space $G\backslash(X*\E)$ with its
quotient topology, and write $[x,a]$ for the orbit of $(x,a)$.

\begin{lemma}
  The map $\px:\ex\to \gux$ given by $\px([x,a])=G\cdot x$ is a
  continuous open surjection.
\end{lemma}
\begin{proof}
  Continuity is clear.  To see that $\px$ is open, we use Fell's
  criterion (see Lemma~\ref{lem-fell-criterion}).  Suppose that
  $G\cdot x_{i}\to G\cdot x= \px([x,a])$.  It suffices to lift a
  subnet to $\ex$.  After passing to a subnet, and relabeling, we can
  assume that $x_{i}\to x$.  Since $r(x_{i})\to r(x)=p(a)$ in $\rgo$
  and since $p$ is open, we can pass to another subnet, relabel, and
  assume that there are $a_{i}\to a$ in $\E$ such that
  $p(a_{i})=r(x_{i})$.  But then $[x_{i},a_{i}]\to [x,a]$ as required.
\end{proof}

To see that $\ex$ is an Banach bundle, we have to equip the fibres
$(\px)^{-1}(G\cdot x)$ with Banach space structures.  The map
$a\mapsto [x,a]$ is a bijection of $E(r(x))$ onto
$(\px)^{-1}(G\cdot x)$.  Thus we can define
$\bigl\|[x,a]\bigr\|=\|a\|$, $[x,a]+[x,b]=[x,a+b]$, and
$\lambda[x,a]=[x,\lambda a]$.
These structures are independent of our choice of representative
$x \in G \cdot x$ and makes $(\px)^{-1}(G\cdot x)$ into a Banach
space.  Note that $E^{X}(G\cdot x):= (\px)^{-1}(G\cdot x)$ is
isomorphic to $E(r(x))$; however, this identification is non-canonical
as it depends on the choice of representative for $G\cdot x$.

\begin{lemma}
  Suppose that $p:\E\to\rgo$ is a (continuous) Banach bundle on which
  $G$ acts by isometric isomorphisms
  $\alpha_{\gamma}:E(s(\gamma))\to E(r(\gamma))$.  Let
  $\px:\ex\to \gux$ be the quotient bundle described above with the
  given Banach space structure on the fibres $E^{X}(G\cdot x)$.  Then
  $\ex$ is a (continuous) Banach bundle over $\gux$.  In particular,
  if $\E$ is a \cs-bundle and $G$ acts by $*$-isomorphisms, then $\ex$
  is a \cs-bundle.
\end{lemma}
\begin{proof}
  The proof proceeds by checking that the axioms for a Banach bundle
  hold just as in \cite{kmrw:ajm98}*{Proposition~2.15}.  The only
  ``upgrade'' from \cite{kmrw:ajm98} is to include the possibility
  that $p$ is merely \usc.  But if $[x_{i},a_{i}]\to [x,a]$ in $\ex$
  with $\|a_{i}\|\ge \epsilon>0$ for all $i$, then we can pass to a
  subnet, relabel, and assume that
  $(\gamma_{i}x_{i}, \alpha_{\gamma_{i}}(a_{i}))\to (x,a)$.  Since
  each $\alpha_{\gamma_{i}}$ is isometric, $\|a\|\ge \epsilon$.
\end{proof}

It will be convenient to describe the sections of a quotient bundle in
terms of the sections of the original bundle.  If $G$ acts
isometrically on $\E$ as above, then given
$\check f\in \Sec(\gux;\ex)$, since the $G$-action on $X$ is
free, there is a function $f:X\to \E$ such that $f(x) \in E(r(x))$ and
\begin{align}
  \label{eq:52}
  \check f(G\cdot x)=[x,f(x)].
\end{align}
Furthermore,
\begin{align}
  \label{eq:53}
  f(\gamma \cdot x)= \alpha_{\gamma}(f(x)).
\end{align}
Now suppose that $x_{i}\to x$.  Then
\begin{align}
  \label{eq:54}
  \check f(G\cdot x_{i})=[x,f(x_{i})]\to \check f(G\cdot x)= [x,f(x)].
\end{align}
Since $\px$ is open, we can pass to a subnet, relabel, and assume that
there are $\gamma_{i}$ such that
\begin{align}
  \label{eq:55}
  \bigl(\gamma_{i}\cdot x_{i}, \alpha_{\gamma_{i}}(f(x_{i})\bigr) \to
  \bigl(x,f(x)\bigr)\quad\text{in $X*\E$.}
\end{align}
Since the $G$-action on $X$ is free and proper, we can assume that
$\gamma_{i}\to s(x)$.  But then since the $G$ action on $X*\E$ is
continuous, $(x_{i},f(x_{i}))\to (x,f(x))$, and in particular,
$f(x_{i})\to f(x)$.  Since we can repeat this argument for any subnet
of $\sset{f(x_{i})}$ it follows that the original net converges to
$f(x)$ and $f:X\to\E$ is continuous.  As a consequence, we have the
following proposition where, as is standard, we have identified
sections $f\in \Sec(X;X*\E)$ of the pullback with continuous functions
$f:X\to\E$ such that $p(f(x))=r(x)$.
\begin{prop}
  \label{prop-sections-ax} The sections in
  $\check f\in \Sec(\gux;\ex)$ are in one-to-one correspondence with
  sections $f\in \Sec(X;X*\E)$ such that
  \begin{equation}
    \label{eq:56}
    f(\gamma \cdot x)=\alpha_{\gamma}\bigl(f(x)\bigr)
    \quad\text{for all $(\gamma,x)\in G*X$. }
  \end{equation}
  We have
  \begin{align}
    \label{eq:57}
    \check f(G\cdot x)=[x,f(x)].
  \end{align}
  In particular, we can identify $\Sec_{c}(\gux;\E^{X})$ with the
  space $\cc(\gux,X,\E,\alpha)$ of continuous functions $f:X\to\E$
  which transform according to \eqref{eq:56}, and such that the
  support of $f$ has compact image in $\gux$.
\end{prop}

We say a net $\set{\check f_{i}}$ in $\Sec_{c}(\gux;\ex)$ converges to
$\check f$ in the \ilt\ if $\check f_{i}\to \check f$ uniformly and
the supports of the $\check f_{i}$ are all contained in a fixed
compact set.  This is equivalent to saying that the corresponding
functions $f_{i} \in \cc(\gux,X,\E,\alpha)$ converge uniformly to $f$
with supports all contained in a set with compact image in $\gux$.

\subsection{Fell bundles and their
  $C^*$-algebras}\label{sec:Fell-bundles-csa}

Fell bundles over groupoids are a natural generalization of Fell's
Banach $*$-algebraic bundles from
\cite{fd:representations2}*{Chapter~VIII}.  They were introduced in
\cite{yam:xx87}.  For more details, see \cites{kum:pams98,
  muhwil:dm08}.  Roughly speaking a Fell bundle $\BB$ over a locally
compact Hausdorff groupoid $\G$ is an upper-semicontinuous Banach
bundle $p : \BB \to \G$ endowed with a continuous involution
$b \mapsto b^*$ and a continuous multiplication $(a,b) \mapsto ab$
from $\BB^{(2)}=\set{(b,b'):(p(b),p(b')\in \G^{(2)}}$ to $\BB$ such
that---with respect to the operations, actions, and inner products
induced by the involution and multiplication---the fibres
$B(x) = p^{-1}(x)$ over units $x \in \G^{(0)}$ are $C^*$-algebras and
such that each fibre $B(\gamma)$ is a
$B(r(\gamma))$--$B(s(\gamma))$-imprimitivity bimodule. We write
$\Sec_c(\G; \BB)$ for the $*$-algebra of continuous compactly
supported sections of $\BB$ under convolution and involution.

Each Fell bundle has both a full and a reduced $C^*$-algebra. The full
$C^*$-algebra $C^*(\G, \BB)$ is universal for representations of the
bundle that are continuous in the inductive-limit topology. The
reduced $C^*$-algebra is obtained from a representation of
$C_c(\G; \BB)$ on the Hilbert module $\H(\BB)$ defined as follows (see
\cites{kum:pams98, moutu:xx11, simwil:nyjm13, hol:jot17}).

The restriction of $\BB$ to a bundle over the unit space is an
upper-semicontinuous $C^*$-bundle, and we write
$\Sec_0(\G^{(0)}; \BB)$ for the $C^*$-algebra of $C_0$-sections of
this bundle. Define $\rip{C_0(\G^{(0)}, \BB)}<\cdot,\cdot>$ on
$\Sec_c(\G; \BB)$ by
$\rip{C_0(\G^{(0)}, \BB)}<f,g> = (f^* * g)|_{\G^{(0)}}$. This is a
pre-inner product, and the completion
\begin{align}\label{eq:HB}
  \H(\BB) := \overline{\Sec_c(\G; \BB)}^{\|\cdot\|_{C_0(\G^{(0)}, \BB)}}
\end{align}
is a Hilbert module. Left multiplication in $\Sec_c(\G; \BB)$ extends
to a left action $\phi : \Sec_c(\G; \BB) \to \Ll(\H(\BB))$ by
adjointable operators, and the reduced norm on $\Sec_c(\G; \BB)$ is
given by $\|f\|_r = \|\phi(f)\|_{\Ll(\H(\BB))}$.

\subsection{Twisted groupoid crossed products}
\label{sec:group-cross-prod}

There are a number of treatments of group\-oid crossed products and
twisted crossed products in the literature.  Here we use a Fell bundle
approach as in \cite{muhwil:dm08}*{\S2}.  Thus a groupoid dynamical
system $(\E,\Sigma, \auto)$ consists of a \cs-bundle $p:\E\to\go$, a
locally compact Hausdorff groupoid $\Sigma$, and isomorphisms
$\auto_{\sigma}:E(s(\sigma))\to E(r(\sigma))$ such that
$\auto_{\sigma\tau}=\auto_{\sigma}\circ \auto_{\tau}$ and such that
$\sigma\cdot e:=\auto_{\sigma}(e)$ is a continuous action of $\Sigma$
on $\E$.  Often in the notation the bundle $\E$ is replaced by the
corresponding $C_{0}(\go)$-algebra $A:=\Sec_{0}(\go;\E)$, and we still
call $(A,\Sigma,\auto)$ a groupoid dynamical system.  As in
\cite{muhwil:dm08}*{Example~2.1}, we can form a Fell bundle
$\BB=r^{*}\E=\set{(a,\sigma):p(a)=r(\sigma)}$ over $\Sigma$ with
multiplication and involution given by
\begin{gather}
  \label{eq:12}
  (a,\sigma)(b,\tau)=\bigl(a\auto_{\sigma}(b),\sigma\tau\bigr)
  \quad\text{and} \quad
  (a,\sigma)^{*}=\bigl(\auto_{\sigma}^{-1}(a),\sigma^{-1}\bigr) .
\end{gather}
Then, if $\Sigma$ has a Haar system, the crossed product
$\cs(\E,\Sigma,\auto)$ is the Fell bundle \cs-algebra
$\cs(\Sigma;\BB)$ built out of $\Sec_{c}(\Sigma;\BB)$.

To get a twisted crossed product, we first require that $\Sigma$ is a
unit space fixing extension of a subgroupoid group bundle $\A$.  That
is, we have
\begin{equation}
  \label{eq:ext}
  \begin{tikzcd}[column sep=3cm]
    \A \arrow[r,"\iota", hook] \arrow[dr,shift left, bend right = 15]
    \arrow[dr,shift right, bend right = 15]&\Sigma \arrow[r,"p", two
    heads] \arrow[d,shift left] \arrow[d,shift right]&\G
    \arrow[dl,shift left, bend left = 15] \arrow[dl,shift right, bend
    left = 15]
    \\
    &\go&
  \end{tikzcd}
\end{equation}
where $\iota$ is the inclusion map, $\G$ is a locally compact
groupoid,  and $p$ is a continuous open
surjection with kernel $\iota(\A)$ restricting to a homeomorphism of
$\go$ and $\goo$.  Notice that we can identify $\G$ with either of the
orbit spaces $\A\backslash \Sigma$ or $\Sigma/\A$ where the
$\A$-action is given by multiplication in $\Sigma$.

Let $U(E(u))\subset M(E(u))$ be the unitary group of $E(u)$, and let
$\coprod_{u\in\go} U(E(u))$ be the corresponding (algebraic) group
bundle over $\go$.  A (Green--Renault) \emph{twisting map} for $\auto$
is a unit-space fixing homomorphism
$\gtw:\A\to \coprod _{u\in\go} U(E(u))$ that induces an action by
isometric Banach space isomorphisms of $\A$ on $\E$ so that
$(a,e)\mapsto a\cdot e:=\gtw(a)e$ is continuous from $\A*\E\to \E$
such that
\begin{align}
  \label{eq:61}
  \auto_{a}(e)&=\gtw(a)e\gtw(a)^{*}\quad\text{for all $(a,e)\in
                \A*\E$, and} \\
  \gtw(\sigma a \sigma^{-1})
              &=\bar\auto_{\sigma}(\gtw(a))
                \quad\text{for all $(\sigma,a)\in \Sigma*\A$.}
\end{align}
Since $e\mapsto e^{*}$ is continuous, so is
$(a,e)\mapsto e\gtw(a)^{*}= (\gtw(a)e^{*})^{*}$.  Then we call
$(\G,\Sigma,\E,\auto,\gtw)$ a Green--Renault twisted groupoid
dynamical system.

Given $(\G,\Sigma,\E,\auto,\gtw)$, we can form a Fell bundle
$\BB=\BB(\E,\auto,\gtw)$ over $\G$ as in
\cite{muhwil:dm08}*{Example~2.5} as follows. Define a left $\A$-action
by isometric isomorphisms on
$r^{*}\E=\Sigma*\E= \set{(\sigma,e)\in\Sigma\times \E: r(\sigma) =
  p_{\E}(e)}$ by
\begin{align}
  \label{eq:10}
  a\cdot (\sigma,e)= (a\sigma,e\gtw(a)^{*}).
\end{align}
Then form a the Banach bundle quotient
$\BB=\BB(\E,\auto,\gtw)=\E^{\Sigma}$ over $\G=\A\backslash \Sigma$ by
forming the quotient $\A\backslash r^{*}\E$.  We will write
$[\sigma,e]$ for the orbit of $(\sigma,e)$ in $\BB$.  Thus
$p_{\BB}:\BB\to\G$ is given by $p_{\BB}([\sigma,e])=\dot \sigma$.

If $(\dot\sigma,\dot\tau)\in\G^{(2)}$, then we can define
\begin{equation}
  \label{eq:62}
  (\sigma,e)(\tau,f)=\bigl(\sigma\tau,e\auto_{\sigma}(f)\bigr),
\end{equation}
and compute that
\begin{align}
  \label{eq:63}
  (a\sigma,e\gtw(a)^{*})(b\tau,f\gtw(b)^{*}) = (a\sigma
  b\sigma^{-1})\cdot \bigl(\sigma\tau,e\auto_{\sigma}(f)\bigr).
\end{align}
Hence \eqref{eq:62} is well defined on elements of $\BB$ and defines a
``multiplication map'' on
$\BB^{(2)}=\set{([\sigma,e],[\tau,f])\in \BB\times\BB:
  (\dot\sigma,\dot\tau)\in \G^{(2)}}$ given by
\begin{equation}
  \label{eq:59}
  [\sigma,e][\tau,f]:= [\sigma\tau, e\auto_{\sigma}(f)].
\end{equation}
Similarly, we get a well-defined involution on $\BB$ defined by
\begin{equation}
  \label{eq:64}
  [\sigma,e]^{*}:= [\sigma^{-1},\auto_{\sigma}^{-1}(e^{*})].
\end{equation}

\begin{lemma}
  With multiplication and involution defined by \eqref{eq:59} and
  \eqref{eq:64}, respectively, $\BB(\E,\auto,\gtw)$ is a Fell bundle
  over $\G$.
\end{lemma}

\begin{proof}
  The continuity of multiplication and the involution follows from the
  openness of the projection $p_{\BB}:\BB\to \G$ and the continuity of
  the action of $\Sigma$ on $\E$.  The algebraic properties of
  multiplication are routine to verify as are axioms (a), (b), and (c)
  of \cite{muhwil:dm08}*{Definition~1.1}.

  For (d)~and~(e), recall that the fibre $B(\dot\sigma)$ of $\BB$ over
  $\dot \sigma\in\G$ is $\set{[\sigma,e]}$ equipped with the
  Banach-space structure induced by the map $e\mapsto [\sigma,e]$ for
  $e \in E(r(\sigma))$.  So for $u\in\go$, the space $B(u)$ is a
  \cs-algebra isomorphic to $E(u)$ with the induced multiplication and
  involution coming from $\BB$. This is~(d).

  For~(e), we must show that $B(\dot\sigma)$ is a
  $B(r(\sigma))\sme B(s(\sigma))$-\ib\ under the actions and inner
  products induced by $\BB$.  These actions and inner products are
  exactly those coming from viewing $E(r(\sigma))$ as a
  $E(r(\sigma))\sme E(s(\sigma)$-\ib\ using the isomorphism
  $\auto_{\sigma}^{-1}:E(r(\sigma))\to E(s(\sigma))$ (see
  \cite{muhwil:dm08}*{Example~2.1}).
\end{proof}

If $\G$ has a Haar system, then the twisted groupoid crossed product
$\csgsgtw$ is the Fell bundle \cs-algebra $\cs(\G;\BB)$ where
$\BB=\BB(\E,\auto,\gtw)$ as above.  Recall that sections
$\check f\in\Sec(\G;\BB)$ are determined by continuous functions
$f:\Sigma\to \E$ such that
\begin{align}
  \label{eq:66}
  f(a\sigma)=f(\sigma)\gtw(a)^{*}\quad\text{for all $(a,\sigma)\in \A*\E$.}
\end{align}
Then
\begin{align}
  \label{eq:67}
  \check f(\dot\sigma)=[\sigma,f(\sigma)].
\end{align}
Thus $\Sec_c(\G; \BB)$ is isomorphic to the space $\ccgsgtw$ of
continuous functions $f:\Sigma\to\E$ that satisfy~\eqref{eq:66} and
whose support has compact image in $\G$, with operations
\begin{align}
  \label{eq:68}
  f*g(\sigma)=\int_{\G} f(\tau)\auto_{\tau}(g(\tau^{-1}\sigma) \,
  d\alpha^{r(\sigma)} (\dot \tau)\quad\text{and}\quad f^{*}(\sigma)=
  \auto_{\sigma} (f(\sigma^{-1})^{*}).
\end{align}
For example,
\[
  \check f* \check g(\dot\sigma) = \int_{\G} \check f(\dot \tau)
  \check g(\dot \tau^{-1}\dot\sigma) \,d\alpha^{r(\sigma)}(\dot \tau)
  = \int_{\G} [\sigma,f(\tau)\auto_{\tau}(g(\tau^{-1}\sigma))] \,d
  \alpha^{r(\sigma)}(\dot \tau),
\]
and since $(\sigma,a)\mapsto [\sigma,a]$ is an isomorphism of Banach
spaces, this gives
\[
  \check f* \check g(\dot\sigma) = \Bigl[\sigma ,\int_{\G}
  f(\tau)\auto_{\tau}(g(\tau^{-1}\sigma) \, d\alpha^{r(\sigma)}(\dot
  \tau)\Bigr].
\]

The $I$-norm on $\ccgsgtw$ is given by
\begin{align}
  \label{eq:69}
  \|f\|_{I}=\max \Bigl\{
  \sup_{u\in\go}\int_{\G}\|f(\sigma)\|\,d\alpha^{u}(\dot \sigma) ,
  \sup_{u\in\go} \int_{\G}\|f(\sigma)\|\,d\alpha_{u}(\dot\sigma)
  \Bigr\}.
\end{align}

Therefore the Fell bundle \cs-algebra $\cs(\G;\BB)$ is the completion
of $\ccgsgtw$ universal for $I$-norm decreasing representations of
$\ccgsgtw$.

\begin{remark}
  If $\Sigma$ is a group $G$ and $\A$ a normal subgroup $N$, then
  $\gtw$ is a twisting map for $(N,G,\auto)$ as in \cite{gre:am78}.
  The extension of twists to groupoid dynamical systems comes from
  \cite{ren:jot87}*{\S3} where it was assumed that $\A$ was an abelian
  group bundle and that $\E$ was continuous.
\end{remark}

For our isomorphism result, we will need the special case of the
twisted crossed product where $\A=\go\times\T$. Such groupoids are
called either $\T$-groupoids or twists.  As in
\cite{muhwil:dm08}*{Example~2.3}, the associated full and reduced
$C^*$-algebras, $\cs(\G;\Sigma)$ and $\cs_{r}(\G;\Sigma)$ coincide
with the full and reduced $C^*$-algebras of the Fell line-bundle $\BB$
over $\G$.  In particular, both these \cs-algebras are completions of
$\C_{c}(\G;\Sigma)$ consisting of $f\in C_{c}(\Sigma)$ such that
$f(z\sigma)=zf(\sigma)$ for all $z\in\T$ and $\sigma\in\Sigma$ (note
that the complex conjugate appearing in equation~(2.3) following
\cite{muhwil:dm08}*{Example~2.3} is a typo).

Hence the reduced norm on $C_c(\G; \Sigma)$ is realized by the
representation of $C_c(\G; \Sigma)$ on the Hilbert module
$\H(\G; \Sigma)$ given by taking the completion of $C_{c}(\G; \Sigma)$
with respect to $\langle f, g\rangle := (f^* * g)|_{\G^{(0)}}$.
Restricting to the special case where the twist $\Sigma$ is the
trivial twist $\G \times \T$, we see that the reduced norm on
$C_c(\G)$ is realized by the representation of $C_c(\G)$ on the
Hilbert module $\H(\G)$ obtained as the completion of $C_c(\G)$ with
respect to $\langle f, g\rangle = (f^* * g)|_{\G^{(0)}}$.

\section{Extensions of group bundles}
\label{sec:extens-abel-group}

Let $\Sigma$ be a second-countable locally compact Hausdorff groupoid
with a Haar system $\lambda=\sset{\lambda^{u}}_{u\in\go}$. Let $\A$ be
a closed subgroupoid of the the isotropy group bundle
$\Sigma':=\set{\sigma\in\Sigma:s(\sigma)=r(\sigma)}$ such that
$\A^{(0)}=\go$. We summarize this by saying that \emph{$\A$ is a wide
  subgroup bundle of $\Sigma$}.  We let $p_{\A}=r\restr\A=s\restr\A$
be the projection from $\A$ onto $\go$ and write $A(u)$ for the fibre
of $\A$ over $u$.

We assume throughout that $\A$ has a Haar system
$\set{\beta^{u}}_{u\in\go}$.  Even though $\Sigma$ is assumed to have
a Haar system, $\A$ has a Haar system only if the fibres of $\A$ are
well-behaved as described in the following result taken from
\cite{ren:jot91}*{\S1} (see also \cite{wil:toolkit}*{Theorem~6.12}).

  \begin{lemma}\label{lem-bundle-haar}
    Let $\A$ be a closed subgroup bundle of $\Sigma$.  Then the
    following are equivalent.
    \begin{compactenum}
    \item There is a Haar system $\beta=\sset{\beta^{u}}_{u\in\go}$
      for $\A$.
    \item The projection map $p_{\A}:\A\to\go$ is open.
    \item The map $u\mapsto A(u)$ is continuous from $\go$ into the
      space of closed subgroups of $\Sigma$ in the Fell topology.
    \end{compactenum}
  \end{lemma}

  Note that $\A$ acts freely and properly on the left as well as on
  the right of $\Sigma$.  Since $p_{\A}$ is open, the orbit maps for
  $\A$-actions are open \cite{wil:toolkit}*{Proposition~2.12}, so both
  orbit spaces $\A\backslash \Sigma$ and $\Sigma/\A$ are locally
  compact Hausdorff in our situation.  We say that $\A$ is
  \emph{normal} if the orbits $\A\sigma$ and $\sigma\A$ coincide for
  all $\sigma\in \Sigma$.  Note that $\A$ is normal if and only if
  $\Sigma$ acts on $\A$ by conjugation.

  If $\A$ is normal, then as subsets of $\Sigma$,
  $(\sigma_{1}\A)(\sigma_{2}\A)= \sigma_{1}\sigma_{2}\A$ whenever
  $s(\sigma_{1})=r(\sigma_{2})$.  This allows us to impose a groupoid
  structure on $\G=\Sigma/\A$ in the obvious way.  Conversely, if
  $p:\Sigma\to\G$ is a continuous, open groupoid homomorphism which
  induces a homeomorphism of $\go$ onto $\goo$, then the kernel is a
  wide normal subgroup bundle and we can identify $\go$ and $\goo$.
  We can summarize these observations as follows where as usual we
  write $\dot \sigma$ in place of $p(\sigma)=\sigma\A$.

  \begin{lemma}
    \label{lem-normal-quotient} Suppose that $\A$ is a wide normal
    subgroup bundle of $\Sigma$ with a Haar system. Consider the orbit
    space $\G= \A\backslash \Sigma =\Sigma/\A$. Define
    $\G^{(2)} = \{(\dot \sigma_1, \dot \sigma_2) : s(\sigma_1) =
    r(\sigma_2)\}$, and define a multiplication map $\G^{(2)} \to \G$
    by $(\dot \sigma_{1})(\dot \sigma_{2})=
    \sigma_{1}\sigma_{2}\A$. Then $\G$ is a locally compact groupoid,
    and $\G^{(0)} \cong \Sigma^{(0)}$ via $u \A \mapsto
    u$. Equivalently, $\A$ is a wide normal subgroup bundle of
    $\Sigma$ if and only if we have a unit space fixing groupoid
    extension \eqref{eq:ext} where $\iota$ is the inclusion map, $p$
    is the orbit map, and we identify $\go$ and $\goo$.
  \end{lemma}

  \begin{remark}
    The isotropy bundle $\A=\Sigma' $ is always a normal subgroup
    bundle.  However, Lemma~\ref{lem-bundle-haar} implies that
    $\Sigma'$ has a Haar system only when the isotropy map
    $u\mapsto \Sigma(u):=\Sigma_{u}^{u}$ is continuous.  Nevertheless,
    there are interesting examples where $\A\subsetneq \Sigma'$ has a
    Haar system even when $\Sigma'$ does not.  A number of such
    instances are given in Section~\ref{sec:etale effective}.
  \end{remark}

  For the remainder of this section, $\Sigma$ will be a second
  countable locally compact Hausdorff groupoid with a Haar system
  $\lambda$, and $\A$ will be a wide normal subgroup bundle with a
  Haar system $\beta$.  We will write $\G$ for the groupoid
  $\Sigma/\A$ as in Lemma~\ref{lem-bundle-haar} and $p:\Sigma\to\G$
  will be the quotient map as in \eqref{eq:ext}.

  The following can be proved using the uniqueness of Haar measure on
  the $A(u)$ exactly as in \cite{mrw:tams96}*{Lemma~4.1}.

  \begin{lemma}
    \label{lem-omega} Let $\Sigma$ and $\A$ be as above.  Then there
    is a continuous homomorphism $\delta:\Sigma\to\R^{+}$, called the
    \emph{modular map for the extension~\eqref{eq:ext}}, such that
    \begin{equation}
      \label{eq:1}
      \int_{\A} f(\sigma a\sigma^{-1})
      \,d\beta^{s(\sigma)}(a)=\delta(\sigma) \int_{\A}
      f(a)\,d\beta^{r(\sigma)} (a).
    \end{equation}
    Note that if $a\in A(u)$, then $\delta(a)=\Delta_{A(u)}(a)$.
  \end{lemma}

  \begin{remark}
    \label{rem-omega-delta} Note that $\delta(\sigma)$ is the inverse
    of the similar constant $\omega(\sigma)$ used in
    \cite{mrw:tams96}, and $\delta$ has the advantage that it
    restricts to the modular function on each $A(u)$.
  \end{remark}

  We let $C_{c,p}(\Sigma)$ be the collection of continuous functions
  on $\Sigma$ such that $\supp b\cap p^{-1}(K)$ is compact for all
  $K\subset \G$ compact.

  \begin{lemma}
    \label{lem-new-bru-haar-gamma} Suppose that $\Sigma$, $\A$, $\G$,
    and $p$ are as above.
    \begin{compactenum}
    \item There is a $b\in C_{c,p}(\Sigma)$, called a \emph{Bruhat
        section}, such that
      \begin{gather}
        \label{eq:29}
        \int_{\A} b(\sigma a)\,d\beta^{s(\sigma)}(a)=1 \quad\text{for
          all $\sigma\in\Sigma$.}
      \end{gather}
    \item There is a surjection $Q:\cc(\Sigma)\to\cc(\G)$ given by
      \begin{equation}
        \label{eq:4}
        Q(f)(\dot \sigma)=\int_{\A} f(\sigma a)\,d\beta^{s(\sigma)}(a).
      \end{equation}
    \item There is a Haar system $\alpha=\sset{\alpha_{u}}_{u\in\go}$
      on $\G$ such that for all $f\in\cc(\Sigma)$ and $u\in\go$ we
      have
      \begin{equation}
        \label{eq:2}
        \int_{\Sigma}f(\sigma)\,d\lambda^{u}(\sigma) = \int_{\G}
        \int_{\A} f(\sigma a)\,d\beta^{s(\sigma)}(a)\, d\alpha^{u}(\dot \sigma).
      \end{equation}
    \end{compactenum}
  \end{lemma}

\begin{proof}
  By left invariance, if $f \in C_c(\Sigma)$ and
  $\dot\sigma = \dot\tau \in \G$ then
  \[\int_\A f(\sigma a)d\beta^{s(\sigma)}(a) = \int_\A f(\tau
    a)d\beta^{s(\tau)}(a).\] Hence for $f \in C_c(\Sigma)$ there is a
  well-defined map $\gamma \mapsto \bbeta^\gamma(f)$ from $\G$ to $\C$
  such that
  \begin{equation}
    \label{eq:3}
    \bbeta^{\dot\sigma}(f) =\int_{\A}f(\sigma a)\,d\beta^{s(\sigma)}(a)
    \quad\text{ for all $\sigma \in \Sigma$.}
  \end{equation}
  The collection $\bbeta=\sset{\bbeta^{\gamma}}_{\gamma\in\G}$ is a
  $p$-system of measures by \cite{wil:toolkit}*{Lemma~2.21(b)}.  The
  existence of $b$ follows from the paracompactness of $\G$ as in
  \cite{wil:toolkit}*{Proposition~3.18}.  This proves~(a) and that
  $Q(f) \in C_{c}(\G)$.  Surjectivity follows from from~(a) and we
  have established~(b).

  (c) To show that a Radon measure $\alpha^{u}$ exists satisfying
  \eqref{eq:2}, it will suffice to see that whenever
  \begin{equation}
    \label{eq:16}
    \int_{\A} f(\sigma a) \,d\beta^{s(\sigma)}(a) =0\quad\text{for
      all $\sigma\in r^{-1}(u)$}
  \end{equation}
  it follows that
  \begin{equation}
    \label{eq:17}
    \int_{\Sigma} f(\sigma)\,d\lambda^{u}(\sigma)=0.
  \end{equation}
  (Then we can define a linear functional on $\cc(\G)$ by
  $\alpha^{u}(Q(f))=\lambda^{u}(f)$.)  But if \eqref{eq:16} holds,
  then for any $h\in\ccs$,
  \begin{align}
    \label{eq:18}
    \int_{\A} f*h(a) \,d\beta^{u}(a)
    &= \int_{\A}\int_{\Sigma} f(a\sigma)h(\sigma^{-1})
      \,d\lambda^{u}(\sigma) \,d\beta^{u}(a) \\
    &= \int_{\Sigma}\Bigl( \delta(\sigma)^{-1} \int_{\A} f(\sigma a)
      \,d\beta^{s(\sigma)} (a) \Bigr) h(\sigma^{-1})
      \,d\lambda^{u}(\sigma) =0.
  \end{align}
  Taking $h$ to be a multiple of $b$ in \eqref{eq:29} by an
  appropriate function in $\ccs$, we can assume that
  \begin{align}
    \label{eq:19}
    \int_{\A}h(\sigma^{-1}a)\,d\beta^{u}(a)=1
  \end{align}
  for all $\sigma\in\supp f$.  Then the left-hand side of
  \eqref{eq:18} is exactly
  \begin{equation}
    \label{eq:20}
    \int_{\Sigma}f(\sigma)\,d\lambda^{u}(\sigma).
  \end{equation}

  If $F\in \cc(\G)$, then
  \begin{equation}
    \label{eq:21}
    \int_{\G}F(\gamma)\,d\alpha^{u}(\gamma) = \int_{\Sigma}
    F(\dot\sigma) b(\sigma) \,d\lambda^{u}(\sigma).
  \end{equation}
  Hence $u\mapsto \alpha^{u}(F)$ is continuous.  For left-invariance,
  suppose that $\eta=\dot\tau$. Then
  \begin{align}
    \label{eq:5}
    \int_{\G} F(\eta\gamma) \,d\alpha^{s(\eta)}(\gamma)
    &= \int_{\Sigma} F(p(\tau\sigma))b(\sigma) \,d\lambda^{s(\tau)}(\sigma) \\
    &= \int_{\Sigma}
      F(p(\sigma))b(\tau^{-1}\sigma)\,d\lambda^{r(\tau)}(\sigma) \\
    &= \int_{\G}F(\gamma) \,d\alpha^{r(\eta)}(\gamma).\qedhere
  \end{align}
\end{proof}

Since $\A$ has a Haar system, we can form the \cs-algebra $\cs(\A)$.
Then it is not hard to verify that $\cs(\A)$ is a $C_{0}(\go)$-algebra
with fibres $\cs(\A)(u)$ identified with the group \cs-algebras
$\cs(A(u))$.  (Since we are working in a groupoid context, we treat
$A(u)$ as a groupoid with a single unit.  This means that the modular
function $\Delta_{A(u)}$ does not appear in the formula for the
involution on $\cc(A(u))$.)  To form a groupoid dynamical system we
need to exhibit $\cs(\A)$ as the section algebra of a \cs-bundle.

Using \cite{wil:crossed}*{Theorem~C.26}, we can equip
\begin{equation}
  \label{eq:6}
  \E:=\coprod_{u\in\go}\cs(A(u))
\end{equation}
with a topology such that it becomes an upper-semicontinuous
\cs-bundle over $\go$ such that each $f\in\cc(\A)$ defines a
continuous section $\check f \in \Sec_{c}(\go;\E)$ given by
$\check f(u)(a)=f(a)$.  Then $\cs(\A)$ can be identified with
$\Sec_{0}(\go;\E)$.

For each $\sigma\in\Sigma$, we get an isomorphism
\begin{equation}
  \label{eq:7}
  \auto_{\sigma}:\cs\bigl(A(s(\sigma))\bigr)\to\cs\bigl(A(r(\sigma))
  \bigr)
\end{equation}
given on $h\in \cc(A(s(\sigma)))$ by
\begin{equation}
  \label{eq:8}
  \auto_{\sigma}(h) (a) = \delta(\sigma)h(\sigma^{-1}a\sigma).
\end{equation}

\begin{prop}[cf., \cite{goe:rmjm12}*{Proposition~2.6}]
  \label{prop-auto} With the notation just established, and writing
  $\auto=\sset{\auto_{\sigma}}$, the triple $(\E,\Sigma,\auto)$ is a
  groupoid dynamical system.
\end{prop}
\begin{proof}
  We have already established that each $\auto_{\sigma}$ is an
  isomorphism, and it is easy to check that
  $\auto_{\sigma\tau}=\auto_{\sigma}\circ \auto_{\tau}$.  It only
  remains to check the continuity of
  $(\sigma, s) \mapsto \auto_{\sigma}(s)$ from $\Sigma*\E$ to $\E$.
  We establish this using \cite{muhwil:nyjm08}*{Lemma~4.3} (see also
  \cite{kmrw:ajm98}*{Lemma~2.13}).

  We must consider the pullbacks
  \begin{equation}
    \label{eq:9}
    r^{*}\bigl(\cs(\A)\bigr)=\Sec_{0}(\Sigma;r^{*}\E)\quad\text{and}
    \quad
    s^{*}\bigl(\cs(\A)\bigr) = \Sec_{0}(\Sigma;s^{*}\E).
  \end{equation}
  But we can clearly identify sections in $\Sec(\Sigma;r^{*}\E)$ with
  continuous functions $f:\Sigma\to \E$ such that
  $p_{\E}(f(\sigma))=r(\sigma)$, and similarly for sections of
  $s^{*}\E$.  Hence we can define
  $\auto:\Sec_{0}(\Sigma;r^{*}\E)\to\Sec_{0}(\Sigma;s^{*}\E)$ by
  $\auto(f)(\sigma)= \auto_{\sigma}\bigl(f(\sigma)\bigr)$.  Then
  $\auto$ is an isomorphism, and the continuity of the action follows
  from \cite{muhwil:nyjm08}*{Lemma~4.3}.
\end{proof}

If $t\in A(u)$, then we can define $\gtw(t):\cc(A(u))\to\cc(A(u))$ by
\begin{equation}
  \label{eq:11}
  \gtw(t)h(a)=\delta(t)^{\half}h(t^{-1}a) = \Delta_{A(u)}(t)^{\half}
  h(t^{-1} a).
\end{equation}
Since $(\gtw(t)h)^{*}*(\gtw(t)k)=h^{*}*k$, it follows that $\gtw(t)$
extends to a unitary in the multiplier algebra
$M\bigl(\cs(A(u))\bigr)$.  It is straightforward to check that
$\gtw(tt')=\gtw(t)\gtw(t')$.

\begin{lemma}
  \label{lem-joint-cty} The action of $\A$ on $\E$ is continuous.
  That is, the map $(t,s)\mapsto \gtw(t)s$ is continuous from
  $\A*\E\to \E$.
\end{lemma}
\begin{proof} Suppose that we have a net
  $(t_{i},s_{i})\to (t_{0},s_{0})$ in $\A*\E$.  We need to show that
  $\gtw(t_{i})s_{i}\to \gtw(t_{0})s_{0}$ in $\E$.  Let
  $u_{i}=p_{\A}(t_{i})=p_{\E}(s_{i})$.  Using
  \cite{wil:crossed}*{Proposition~C.20}, it will suffice to show that
  given $\epsilon>0$, there are $s'_{i}\to s'_{0}$ in $\E$ such that
  $p_{\E}(s'_{i})=u_{i}$ and that we eventually have
  $\|s'_{i}-\gtw(t_{i})s_{i}\|<\epsilon$.

  Let $f\in \cc(\A)$ be such that
  $\|\check f(u_{0})-s_{0}\|<\epsilon$.  Then we eventually have
  $\|\check f(u_{i})-s_{i}\|<\epsilon$.  Since $\A$ acts by unitaries,
  we eventually have
  $\|\gtw(t_{i})\check f(u_{i}) -\gtw(t_{i})s_{i}\|<\epsilon$.

  Therefore it will suffice to show that
  $s'_{i}:=\gtw(t_{i})\check f(u_{i}) \to s'_{0}:=\gtw(t_{0})\check
  f(u_{0})$ in $\E$.  For this, we form the pullback bundle
  $p_{\A}^{*}\E=\set{(t,s)\in \A\times \E:p_{\A}(t)=p_{\E}(s)}$.  If
  $F\in \cc(\A*\A)$, then we get a section of $p_{\A}^{*}\E$ given by
  $\check F(t)=(t,F(t,\cdot))\in \cc(A(p_{\A}(t))$.  If
  $F(t,a)=f(t)g(a)$ for $f,g\in \cc(\A)$, then we clearly have
  $\check F\in \Sec(\A;p_{\A}^{*}\E)$.  Since finite sums of such
  functions are dense in the inductive-limit topology on $\cc(\A*\A)$,
  we have $\check F\in \Sec(\A;p_{\A}^{*}\E)$ for all
  $F\in \cc(\A*\A)$.

  We can assume that all the $t_{i}$ are in a compact neighborhood $V$
  of $t_{0}$.  Then if $\phi\in\cc^{+}(\A)$ is such that
  $\phi\equiv 1$ on $V$, then $F(t,a) = \phi(t)\gtw(t)f(a)$ defines an
  element of $\cc(\A*\A)$ and
  $\check F(t_{i})=(t_{i},\gtw(t_{i})\check f(u_{i}))$.  Since
  projection on the second factor is continuous from $p_{\A}^{*}\E$ to
  $\E$, the result follows.
\end{proof}

Since $h\gtw(t)^{*}(a)=\delta(t)^{\half}h(at)$, it is routine to check
that $\auto_{t}(s) =\gtw(t) s \gtw(t)^{*}$ for all $(t,s)\in\A*\E$,
and that
$\gtw(\sigma t \sigma^{-1})=
\overline{\auto}_{\sigma}\bigl(\gtw(t)\bigr)$ for all
$(\sigma,t)\in \Sigma*\A$. It follows from Lemma~\ref{lem-joint-cty}
that $\gtw$ is a twisting map for $(\E,\Sigma,\auto)$.  Then we can
form the twisted groupoid crossed product $\csgsgtw$.  Recall that as
described in Section~\ref{sec:group-cross-prod}, the later is the Fell
bundle \cs-algebra for the Fell bundle
$\BB(\E,\auto,\gtw)=\E^{\Sigma}$ associated to the twist.  Hence
$\BB(\E,\auto,\gtw)$ is the quotient of
$r^{*}\E=\set{(\sigma,a)\in \Sigma\times\A:r(\sigma)=p_{\E}(a)}$ by
the $\A$-action $t\cdot (\sigma,a)=(t\sigma, a\gtw(t)^{*})$.

We can identify $\Sec_{c}(\G;\BB(\E,\auto,\gtw))$ with the collection
$\ccgsgtw$ of continuous functions $f:\Sigma\to \E$ such that
$p_{\E}(f(\sigma))=r(\sigma)$ and
\begin{equation}
  \label{eq:70}
  f(a\sigma)=f(\sigma)\gtw(a)^{*}
\end{equation}
whose support has compact image in $\G$.  The $*$-algebra structure on
$\ccgsgtw$ is given by
\begin{align}
  \label{eq:68a}
  f*g(\sigma)=\int_{\G} f(\tau)\auto_{\tau}(g(\tau^{-1}\sigma)) \,
  d\alpha^{r(\sigma)} (\dot \tau)\quad\text{and}\quad f^{*}(\sigma)=
  \auto_{\sigma} (f(\sigma^{-1})^{*}).
\end{align}

If $f\in\cc(\Sigma)$ and $\sigma\in\Sigma$, we let $\renj(f)(\sigma)$
be the element of $\cc(A(r(\sigma)))$ given by
$a\mapsto \delta(\sigma)^{\half}f(a\sigma)$.  In particular,
$j(f)\in \Sec_{c}(\Sigma;r^{*}\E)$ and a quick computation verifies
that it satisfies \eqref{eq:70} and is an element of $\ccgsgtw$.

\begin{remark}\label{rem-ilt}
  Let $\sset{g_{i}}$ be a net in $\ccgsgtw$.  We say that $g_{i}\to g$
  in the inductive-limit topology if $g_{i}\to g$ uniformly and the
  supports of the $g_{i}$ are all contained in some $B$ such that
  $p(B)$ is compact in $\G$.  This implies that
  $\check g_{i}\to\check g$ in the inductive-limit topology on
  $\Sec_{c}(\G;\BB(\E,\auto,\gtw))$.  Now suppose that $f_{i}\to f$ in
  the inductive-limit topology on $\Sigma$ so that there is a compact
  set $K\subset \Sigma$ such that $\supp f_{i}\subset K$ for all $i$
  and $f_{i}\to f$ uniformly.  We claim that $j(f_{i})\to j(f)$ in the
  inductive-limit topology on $\ccgsgtw$.  Certainly we have the
  supports of the $j(f_{i})$ all contained in the image of $K$.
  Moreover if $\sigma \in K$, then
  $\supp j(f_{i})(\sigma)\subset K^{-1}K\cap \A$.  Since
  $u\mapsto \beta^{u}(K^{-1}K\cap \A)$ is bounded and $\delta$ is
  bounded on $K$, it follows that $j(f_{i})\to j(f)$ uniformly on $K$.
  Since $\|j(f_{i})(\sigma) - j(f)(\sigma)\|$ depends only on
  $\dot\sigma$, the claim follows.
\end{remark}

\begin{lemma}
  \label{lem-homo} The map $f\mapsto \renj(f)$ is a $*$-homomorphism
  of $\cc(\Sigma)$ into $\ccgsgtw$ and the range of $\renj$ is dense
  in $\ccgsgtw$ in the inductive limit topology.
\end{lemma}
\begin{proof}
  Using \eqref{eq:68a} we have
  \begin{align}
    \label{eq:25}
    \renj(f*g)(\sigma)(a')
    &= \delta(\sigma)^{\half} f*g(a'\sigma) \\
    &= \delta(\sigma)^{\half} \int_{\Sigma} f(\tau) g(\tau^{-1} a'\sigma)
      \,d\lambda^{r(\sigma)} (\tau) \\
    &= \delta(\sigma)^{\half}\int_{\G}\int_{\A} f(\tau a)
      g(a^{-1}\tau^{-1} a'\sigma) \,\beta^{s(\tau)}(a)
      \,d\alpha^{r(\sigma)} (\dot \tau) \\
    &= \delta(\sigma)^{\half} \int_{\G}\delta(\tau) \int_{\A}
      f(a\tau) g(\tau^{-1} a^{-1} a' \sigma) \,d\beta^{r(\tau)}(a)
      \,d\alpha^{r(\sigma)} (\dot \tau) \\
    &=\int_{\G}\int_{\A} \delta(\tau)^{\half}f(a\tau)
      \delta(\tau) \delta(\tau^{-1}\sigma)^{\half} g(\tau^{-1} a^{-1}a
      ' \sigma) \,d\beta^{r(\tau)}(a) \,d\alpha^{r(\sigma)}(\dot \tau)
    \\
    &=\int_{\G}\int_{\A} \renj(f)(\tau)(a)
      \auto_{\tau}\bigl(\renj(g)(\tau^{-1}\sigma)\bigr) (a^{-1}a')
      \,\beta^{r(\tau)} (a) \,d\alpha^{r(\sigma)}(\dot\tau) \\
    &=\int_{\G} \renj(f)(\tau)*
      \auto_{\tau}\bigl(\renj(g)(\tau^{-1}\sigma)\bigr)(a')
      \,d\alpha^{r(\sigma)}(\dot\tau) \\
    \intertext{which, arguing as in \cite{wil:crossed}*{Lemma~1.108}, is}
    &= \Bigl(\int_{\G}
      \renj(f)(\tau)*\auto_{\tau}\bigl(\renj(g)(\tau^{-1}\sigma) \bigr)
      \,d\alpha^{r(\sigma)} (\dot\tau)\Bigr)(a').
  \end{align}
  Thus $\renj(f*g)(\sigma)=\renj(f)*\renj(g)(\sigma)$ as required.

  Similarly,
  \begin{align}
    \label{eq:26}
    \renj(f^{*})(\sigma)(a)
    &= \delta(\sigma)^{\half} f^{*}(a\sigma)
      = \delta(\sigma)^{\half}
      \overline{f(\sigma^{-1}a^{-1}\sigma
      \sigma^{-1})} \\
    &= \overline{\auto_{\sigma}\bigl(\renj(f)(\sigma^{-1})\bigr)(a^{-1})}
      = \bigl(\auto_{\sigma}\bigl(\renj(f)(\sigma^{-1}) \bigr)\bigr)^{*}(a).
  \end{align}
  Thus $\renj(f^{*})=\renj(f)^{*}$, and we have shown that $\renj$ is
  a $*$-homomorphism.

  To see that the range is dense, let $F\in \ccgsgtw$.  Let $K$ be a
  compact neighborhood of $p(\supp F)$.  Given $\epsilon>0$ it will
  suffice to find $f\in C_{c}(\Sigma)$ such that
  $p(\supp \renj(f))\subset K$ and $\|\renj(f)-F\|_{\infty}<\epsilon$.

  Given $\sigma\in\Sigma$, there exists $f_{\sigma}\in\cc(\Sigma)$
  such that
  \begin{equation}
    \label{eq:34}
    \|\renj(f_{\sigma})(\sigma)-F(\sigma)\|<\epsilon.
  \end{equation}
  Since $r^{*}\E$ is an upper-semicontinuous Banach bundle and $\gtw$
  is unitary-valued, there is an open neighborhood, $V_{\dot\sigma}$
  of $\dot \sigma$ in $\G$ such that
  \begin{equation}
    \label{eq:35}
    \|\renj(f_{\sigma})(\tau)-F(\tau)\|<\epsilon\quad\text{for all $\dot
      \tau\in
      V_{\dot\sigma}$.}
  \end{equation}
  Fix $\sigma_{1},\dots,\sigma_{n}\in\Sigma$ such that the
  $V_{\dot\sigma_{k}}$ cover $p(\supp F)$.  Let
  $\sset{\phi_{k}} \subset \cc^{+}(\G)$ be such that
  $\supp \phi_{k}\subset V_{\dot\sigma_{k}}\cap K$ and
  \begin{align}
    \label{eq:36}
    \sum_{k=1}^{n}\phi_{k}(\dot \tau)=1
  \end{align}
  if $\tau\in p(\supp F)$ and bounded by $1$ otherwise.  Then
  \begin{equation}
    \label{eq:37}
    f(\sigma) = \sum_{k=1}^{n}\phi_{k}(\dot \sigma)f_{\sigma_{k}}(\sigma)
  \end{equation}
  belongs to $\cc(\Sigma)$, $p(\supp \renj(f) )\subset K$, and
  \begin{align}
    \label{eq:38}
    \|j(f)(\sigma)-F(\sigma)\| \le \sum_{k=1}^{n}\bigl\|\phi_{k}(\dot\sigma)
    \bigl( \renj(f_{\sigma_{k}})(\sigma)-F(\sigma)\bigr) \bigr\| <\epsilon
    \sum_{k=1}^{n} \phi_{k}(\dot\sigma)\le \epsilon.
  \end{align}
  This suffices.
\end{proof}


\begin{thm}
  \label{thm-hm} We let $\Sigma$, $\A$, $\G$, $\E$, and $\auto$ be as
  above with twisting map $\gtw$.  The $*$-homomorphism
  $\renj:\cc(\Sigma)\to \ccgsgtw$ defined by
  $\renj(f)(\sigma)(a)=\delta(\sigma)^{\half}f(a\sigma)$ in
  Lem\-ma~\ref{lem-homo} is isometric for the (universal) \cs-norm and
  therefore extends to an isomorphism
  $\renj : \cs(\Sigma) \to \csgsgtw$.  If $\A$ is amenable, then $j$
  is also isometric for reduced norms and extends to an isomorphism
  $\renj_{r} : \cs_r(\Sigma) \to \csrgsgtw$.
\end{thm}

\begin{remark}
  If $\A$ is not amenable, the situation is complicated. One might
  expect to replace $C^*(\A)$ with $C^*_r(\A)$ in the construction of
  $\E$ to obtain a bundle $\E_r$ and an isomorphism
  $C^*_r(\Sigma) \cong C^*_r(\G, \Sigma, \E_r, \kappa_r)$. However, as
  shown in \cites{wil:mjm15, arm:phd19}, while $C^*_r(\A)$ is a
  $C_0(\G^{(0)})$-algebra, its fibre over $u$ need not be isomorphic
  to $C^*_r(A(u))$. We have not pursued this subtlety further.
\end{remark}

  \begin{remark}\label{rem-bm}
    With some effort, the assertion that $\cs(\Sigma)$ and $\csgsgtw$
    are isomorphic can be derived from a general result in a 2016
    preprint due to Buss and Meyer \cite{bm:xx16}.  The map
    $p:\Sigma\to \G$ from \eqref{eq:ext} is an example of what is
    called a ``groupoid fibration with fibre $\A$'' in \cite{bm:xx16}.
    The Fell bundle constructed in \cite{bm:xx16}*{\S6} can be shown
    to be isomorphic (as Fell bundles) to the Fell bundle
    $\BB( \E,\auto,\gtw)=\E^{\Sigma}$ corresponding to $\csgsgtw$
    constructed above.  Then one can show that $j$ is the composition
    of the isomorphism from \cite{bm:xx16}*{Theorem~6.2} and the
    isomorphism induced by the Fell bundle isomorphism above.  Since
    our situation is considerably more concrete, we are able to
    explicitly describe both the bundle and the isomorphism.  This
    will be crucial in the next section where we specialize to the
    case $\A$ is an abelian group bundle.  Furthermore, the case of
    the reduced norm is not considered in \cite{bm:xx16}.
  \end{remark}

  To prove Theorem~\ref{thm-hm} we will need the following technical
  result from \cite{ren:jot91}*{Corollary~1.8}. We have included the
  details for completeness.

\begin{lemma}\label{lem-modular}
  If $\mu$ is a quasi-invariant measure on $\go$ with respect to
  $\Sigma$, then $\mu$ is also quasi-invariant with respect to $\G$.
  If $\Deltab$ is a modular function on $\G$ for $\mu$, then
  $\Delta(\sigma)=\delta(\sigma)\Deltab(\dot \sigma)$ is a modular
  function on $\Sigma$ for $\mu$.  In particular, we can assume both
  $\Deltab$ and $\Delta$ are homomorphisms into $\R^{+}$.
\end{lemma}
\begin{proof}
  Let $b\in C_{c,p}^{+}(\Sigma)$ be a Bruhat section as in
  Lemma~\ref{lem-new-bru-haar-gamma}(a).  Suppose $f\in\cc(\G)$.  Then
  \begin{align}
    \nu_{\G}(f)
    &=\int_{\go}\int_{\G}
      f(\lambda)\,d\alpha^{u}(\lambda) \,d\mu(u) \\
    &=\int_{\go}\int_{\G}\int_{\A} f(\dot \sigma )b(\sigma a)\,
      d\beta^{s(\sigma)} (a)\,d\alpha^{u}(\dot\sigma) \,d\mu(u) \\
    &= \int_{\go}\int_{\Sigma} f(\dot\sigma) b(\sigma)
      \,d\lambda^{u}(\sigma) \,d\mu(u).
  \end{align}
  So for any modular function $\Delta$ for $\mu$ on $\Sigma$, we
  obtain
  \begin{align}
    \nu_\G(f)
    &= \int_{\go}\int_{\Sigma} f(\dot\sigma^{-1}) b(\sigma^{-1})
      \Delta(\sigma^{-1}) \,d\lambda^{u}(\sigma)\,d\mu(u) \nonumber\\
    &=\int_{\go}\int_{\G}f(\dot\sigma^{-1})\int_{\A}
      b( a^{-1}\sigma^{-1}) \Delta( a^{-1}\sigma^{-1})
      \,d\beta^{s(\sigma)} (a)\,d\alpha^{u}(\dot\sigma)\,d\mu(u).\label{eq:nugf}
  \end{align}
  Define $B : \G \to \C$ by
  \[
    B(\dot\sigma)
    =\int_{\A}b(a^{-1}\sigma^{-1})\Delta(a^{-1}\sigma^{-1})\,
    d\beta^{s(\sigma)}(a).
  \]
  Then~\eqref{eq:nugf} gives
  \[
    \nu_\G(f) = \int_{\go}\int_{\G} f(\gamma^{-1}) B(\gamma)
    \,d\sigma^{u}(\gamma) \,d\mu(u) = \nu_{\G}^{-1}(f B^{*}).
  \]
  Since $\delta$ agrees with $\Delta_{\A(u)}$, we have
  \begin{align}
    B(\dot\sigma)
    &= \delta(\sigma)^{-1}\int_{\A}b(\sigma^{-1}a^{-1})
      \Delta(\sigma^{-1}a^{-1}) \,d\beta^{r(\sigma)}(a)\\
    &= \int_{\A} b(\sigma^{-1}a) \Delta(\sigma^{-1}a)
      \delta(\sigma^{-1}a^{-1}) \,d\beta^{r(\sigma)}(a).
  \end{align}

  Since $\Delta$ and $\delta$ never vanish, it follows from
  \eqref{eq:29} that $B$ never vanishes. Hence $\nu_{\G}$ and
  $\nu_{\G}^{-1}$ are equivalent.  Thus, by definition, $\mu$ is
  $\G$-quasi-invariant.

  {\allowdisplaybreaks 
    Let $\Deltab$ be a modular function for $\G$ (with respect to
    $\mu$). Then
    \begin{align}
      \label{eq:45}
      \int_{\go}&\int_{\Sigma} f(\sigma^{-1})
                  \Deltab(\dot\sigma^{-1})\delta(\sigma) ^{-1}
                  \,d\lambda^{u}(\sigma) \,d\mu(u) \\
                &=
                  \int_{\go}\int_{\G}\Deltab(\dot\sigma^{-1})
                  \delta(\sigma)^{-1}\int_{\A}
                  f( a^{-1}\sigma^{-1})\delta(a)^{-1}
                  \,d\beta^{s(\sigma)}(a)
                  \,d\alpha^{u}(\dot\sigma) \,d\mu(u) \\
                &= \int_{\go}\int_{\G} Q(f)(\dot\sigma^{-1})
                  \Deltab(\dot\sigma^{-1})
                  \,d\alpha^{u}( \dot\sigma)\,d\mu(u) \\
                &= \int_{\go}\int_{\G} Q(f)(\dot\sigma)
                  \,d\alpha^{u}(\dot\sigma)\,d\mu(u)\\
                &= \int_{\go}\int_{\Sigma}
                  f(\sigma)\,d\lambda^{u}(\sigma) \,d\mu(u).
    \end{align}
    Hence $\Deltab(\dot\sigma)\delta(\sigma)$ is a modular function
    for $\Sigma$.  Work of Ramsay---see
    \cite{wil:toolkit}*{Proposition~7.6}---implies we can take
    $\Deltab$ to be a homomorphism.  Then we can let
    $\Delta(\sigma)=\Deltab(\dot \sigma) \delta(\sigma)$.  This
    completes the proof.} 
\end{proof}

\begin{proof}[Proof of Theorem~\ref{thm-hm}]
  Suppose that $f_{i}\to f$ in the \ilt\ on $C_c(\Sigma)$.  Then as in
  Remark~\ref{rem-ilt}, $\renj(f_{i})\to\renj(f)$ in the \ilt\ on
  $\ccgsgtw$.  Thus if $\Rep$ is a nondegenerate representation of
  $\Sec_{c}(\G;\BB(\E,\auto,\gtw))$, then $\Rep\circ j$ is continuous
  in the \ilt\ on $\cc(\Sigma)$ and therefore bounded with respect to
  the \cs-norm:
  \begin{equation}
    \label{eq:22}
    \|\Rep\bigl(j(f)\bigr) \|\le\|f\| \quad\text{for all
      $f\in\cc(\Sigma)$.}
  \end{equation}
  Since $\Rep$ is arbitrary, $\renj$ extends to a homomorphism
  $\renj:\cs(\Sigma) \to \csgsgtw$.

  Proving that $\renj$ is isometric for universal norms requires
  considerably more work. The idea is straightforward. Given a
  nondengenerate representation $L$ of $\cc(\Sigma)$, it will suffice
  to produce a representation $\Rep$ of
  $\Sec_{c}(\G;\BB(\E,\auto,\gtw))$ such that
  $L(f)=\Rep(\renj(f))$. Applying this to a faithul $L$ will show that
  $j$ is isometric.

  Using Renault's Disintegration Theorem (cf., e.g.,
  \cite{wil:toolkit}*{Theorem~8.2}), we can assume that $L$ is the
  integrated form of a unitary representation $(\mu,\go*\HH,\hat L)$
  where $\mu$ is a quasi-invariant measure on $\go$, $\go*\HH$ is a
  Borel Hilbert bundle over $\go$, and $\hat L$ is a groupoid
  homomorphism of $\Sigma$ into $\Iso(\go*\HH)$ of the form
  $\hat L(\sigma)=\bigl(r(\sigma),L_{\sigma},s(\sigma)\bigr)$. By
  Lemma~\ref{lem-modular}, $\mu$ is quasi-invariant with respect to
  $\G$.  Furthermore, we may assume that $\Delta$ is given by
  $\Delta(\sigma) = \delta(\sigma) \Deltab(\dot\sigma)$ where $\delta$
  is given by Lemma~\ref{lem-omega}, and $\Deltab$ is a modular
  function for $\G$.  We can also assume that $\Delta$ and $\Deltab$
  are homomorphisms.

  We are going to realize $\Rep$ as the integrated form of a Borel
  $*$-functor $\hat\pi:\BB(\E,\auto,\gtw)\to\End(\go*\HH)$ as in
  \cite{muhwil:dm08}*{Definition~4.5 and Proposition~4.10}.  To start,
  consider $\sigma\in\Sigma$ and $h\in\cc(A(r(\sigma)))$.  Define
  $\bpi_{\sigma}(h):\H(s(\sigma))\to \H(r(\sigma))$ by
  \begin{equation}
    \label{eq:24}
    \bpi_{\sigma}(h)\xi=\int_{\A}h(a)L_{a\sigma}(\xi)\delta(a)^{-\half}
    \,d\beta^{r(\sigma)} (a).
  \end{equation}
  It is routine to check that
  \begin{equation}
    \label{eq:27}
    \bpi_{t\sigma}(h)= \bpi_{\sigma}(h\gtw(t))\quad\text{for
      $(t,\sigma)\in\A*\Sigma$.}
  \end{equation}
  Since
  \begin{equation}
    \label{eq:28}
    \bip(\bpi_{\sigma}(h)\xi|\eta)= \int_{\A} h(a)
    \delta(a)^{-\half}\bip(L_{a\sigma}(\xi)|\eta)\,d\beta^{r(\sigma)},
  \end{equation}
  it follows from the usual Cauchy--Schwarz estimate that
  \begin{align}
    \label{eq:30}
    \bigl| \bip(\bpi_{\sigma}(h)\xi|\eta) \bigr|^{2}
    &\le
    \|\xi\|^{2}\|\eta\|^{2} \Bigl( \int_{\A}|h(a)|
      \delta(a)^{-\half}\,d\beta^{r(\sigma)}(a)\Big)^{2} \\
    &\le
    \|\xi\|^{2}\|\eta\|^{2} 
    \|h\|_{I,r}\|h\|_{I,s}
    \le  \|\xi\|^{2}\|\eta\|^{2}\|h\|_{I}^{2}.
  \end{align}
  Therefore $\bpi_{\sigma}$ extends to all of $\cs(A(r(\sigma)))$ and
  still satisfies \eqref{eq:27}.  Hence we get a map
  $\pi:\BB(\E,\auto,\gtw) \to\End(\go*\HH)$ defined by
  \begin{align}
    \label{eq:31}
    \pi([\sigma,s])=\bpi_{\sigma}(s).
  \end{align}
  It is not hard to check that $\pi$ determines a $*$-functor
  $\hat \pi(b)=(r(b),\pi(b),s(b))$.  To see that $\hat\pi$ is Borel we
  need to show that
  \begin{equation}
    \label{eq:13}
    \dot\sigma\mapsto \bip(\pi(\check
    f(\dot\sigma))(h(s(\sigma)))|{k(r(\sigma))})
  \end{equation}
  is Borel for all $\check f\in \Sec(\G;\BB(\E,\auto,\gtw)))$ and all
  Borel sections $h$ and $k$ of $\go*\HH$.  By Lemma~\ref{lem-homo},
  we can assume that $\check f$ is defined by $f\in \ccgsgtw$ in the
  range of $\renj$.  Thus $f(\sigma)\in \cc(A(r(\sigma)))$ for all
  $\sigma\in\Sigma$.  Then
  \begin{align}
    \label{eq:14}
    \bip(\pi(\check
    f(\dot\sigma))(h(s(\sigma)))|{k(r(\sigma))})
    &= \bip(\pi_{\sigma}(f(\sigma))(h(s(\sigma)))| {k(r(\sigma))}) \\
    &= \int_{\A} f(\sigma)(a) \bip(L_{a\sigma}(h(s(\sigma)))|
      {k(r(\sigma))}) \delta(a)^{-\half} \,d\beta^{r(\sigma)}(a),
  \end{align}
  which is Borel since $L$ is. Let $\Rep$ be the integrated form of
  $\hat \pi$.

  Take $f\in\cc(\Sigma)$ and $\xi\in L^{2}(\go*\HH,\mu)$.  Then
  \begin{align}
    \label{eq:33}
    L(f)\xi(u)
    &= \int_{\Sigma} f(\sigma) L_{\sigma}\bigl(\xi(s(\sigma))\bigr)
      \Delta(\sigma)^{ -\half}\,d\lambda^{u}(\sigma) \\
    &= \int_{\G}\int_{\A} f(\sigma a) L_{\sigma a}
      \bigl(\xi(s(\sigma))\bigr) \Delta(\sigma a)^{-\half}
      \,d\beta^{s(\sigma)}(a) \,d\alpha^{u}(\dot \sigma) \\
    &= \int_{\G}\delta(\sigma)\int_{\A} f(a\sigma)
      L_{a\sigma}\bigl(\xi(s (\sigma))\bigr) \Delta(a\sigma)^{-\half}
      \,d\beta^{u}(a) \,d\alpha^{u}(\dot \sigma) \\
    &= \int_{\G}\delta(\sigma)^{\half} \int_{\A}
      \renj(f)(\sigma)(a) L_{a\sigma}\bigl( \xi(s(\sigma))\bigr)
      \Deltab(\dot \sigma)^{-\half} \delta(a\sigma)^{-\half}
      \,d\beta^{u}(a) \,d\alpha^{u}(\dot \sigma) \\
    &=\int_{\G} \pi(\renj(f(\gamma)))(h(s(\gamma)))
      \Deltab(\gamma)^{-\half}(\gamma) \\
    &=\Rep(\renj(f))\xi(u).
  \end{align}
  This shows that $\renj$ is isometric.

  The surjectivity of $\renj$ follows from Lemma~\ref{lem-homo}.  This
  finishes the proof that $\renj$ extends to an isomorphism
  $\renj : \cs(\Sigma) \to \csgsgtw$.

  It remains to prove that $j$ is isometric for reduced norms when
  $\A$ is amenable. First, note that by
  \cite{hol:jot17}*{Example~3.14} the space $C_c(\Sigma)$ completes to
  a $C^*(\Sigma)$--$C^*(\A)$ correspondence $X(\Sigma)$ with actions
  given by convolution and inner product
  $\langle f, g\rangle_{C^*(\A)} = (f^* * g)|_{\A}$. Let $\BB$ denote
  the Fell bundle $\BB(\E, \vartheta, \kappa)$ described above so that
  $\csgsgtw$ is the $C^*$-algebra $C^*(\G, \BB)$ of this bundle, and
  similarly for reduced $C^*$-algebras. Let $\H(\BB)$ be the
  right-Hilbert $C^*(\A)$-module of
  Section~\ref{sec:Fell-bundles-csa}, so that the left action of
  $C_c(\G; \BB)$ on $\H(\BB)$ determined by multiplication in
  $C_c(\G; \BB)$ is isometric for the reduced norm. By construction of
  these maps, the map $\renj : C_c(\Sigma) \to \csgsgtw$ extends to a
  right-$C^*(\A)$-module isomorphism $\rho : X(\Sigma) \to \H(\BB)$,
  which satisfies
  \begin{equation}\label{eq:left actions}
    \rho(f \cdot \xi) = j(f) \cdot \rho(\xi)\quad\text{ for $f \in
      C_c(\Sigma)$ and $ \xi \in X(\Sigma)$.}
  \end{equation}

  Let $\H(\Sigma)$ and $\H(\A)$ be the Hilbert modules described in
  Section~\ref{sec:Fell-bundles-csa} that carry faithful
  representations of $C^*_r(\Sigma)$ and $C^*_r(\A)$ respectively. We
  can form the modules $X(\Sigma) \otimes_{C^*(\A)} \H(\A)$ and
  $\H(\BB) \otimes_{C^*(\A)} \H(\A)$, and the isomorphism
  $\rho : X(\Sigma) \to \H(\BB)$ defined above determines an
  isomorphism
  $\rho \otimes \id : X(\Sigma) \otimes_{C^*(\A)} \H(\A) \to \H(\BB)
  \otimes_{C^*(\A)} \H(\A)$ which again intertwines the left actions
  similarly to~\eqref{eq:left actions}. Since the bundle $\A$ is an
  amenable groupoid, the left action of $C^*(\A)=\cs_{r}(\A)$ on
  $\H(\A)$ is faithful, and it follows that the map
  $T \mapsto T \otimes 1$ from $\Ll(\H(\BB))$ to
  $\Ll(\H(\BB) \otimes_{C^*(\A)} \H(A))$ is isometric. So for
  $f \in C_c(\Sigma)$, we have
  \[
    \|\rho(f) \otimes 1\|_{\Ll(\H(\BB) \otimes_{C^*(\A)} \H(A))} =
    \|\rho(f)\|_{\csrgsgtw}.
  \]
  We will show that the map $f \otimes a \mapsto f \cdot a$ extends to
  an isomorphism $X(\Sigma) \otimes_{C^*(\A)} \H(A)$ to
  $\H(\Sigma)$. This will complete the proof since then
  \[
    \|f\|_{C^*_r(\Sigma)} = \|f \otimes 1\|_{\Ll(X(\Sigma)
      \otimes_{C^*(\A)} \H(A))} = \|\rho(f) \otimes 1\|_{\Ll(\H(\BB)
      \otimes_{C^*(\A)} \H(A))} = \|\rho(f)\|_{\csrgsgtw}
  \]
  for all $f \in C_c(\Sigma)$.

  To see that $f \otimes a \mapsto f \cdot a$ extends to the desired
  isomorphism, we fix $f,g \in C_c(\Sigma)$ and $a,b \in C_c(\A)$, and
  calculate:
  \begin{align*}
    \langle f \otimes a, g \otimes b\rangle_{C_0(\Sigma^{(0)})}
    &= \big(a^* * (f^* * g)|_{\A} *
      b\big)|_{C_0(\Sigma^{(0)})}\quad\text{ and}\\
    \langle f \cdot a, g\cdot b\rangle_{C_0(\Sigma^{(0)})}
    &= \big(a^* * (f^* * g) * b\big)|_{C_0(\Sigma^{(0)})}.
  \end{align*}
  For any $h \in C_c(\Sigma)$, and any $x \in \Sigma^{(0)}$, we have
  \[
    (a^* * h * b)(x) = \int_{\Sigma} \int_{\Sigma} a^*(\beta)
    h(\gamma)
    b((\beta\gamma)^{-1}x)\,d\lambda^{x}(\beta)\,d\lambda^{s(\beta)}(\gamma).
  \]
  Since $a,b \in C_c(\A)$, the integrand is nonzero only when
  $\beta, \beta^{-1}\gamma \in \A_x$, and this forces
  $\gamma \in \A_x$. So
  \[
    \big(a^* * (f^* * g) * b\big)|_{\Sigma^{(0)}} = \big(a^* * (f^* *
    g)|_{\A} * b\big)|_{\Sigma^{(0)}}.
  \]
  This completes the proof of Theorem~\ref{thm-hm}.
\end{proof}

\section{The abelian case}
\label{sec:abelian-case}

In this section, we specialize to the case where $\A$ is a normal
\emph{abelian} subgroup bundle of $\Sigma$ with a Haar system $\beta$
as above.  Then $\cs(\A)$ is a commutative \cs-algebra.  Let $\hA$ be
the Gelfand dual space of nonzero complex homomorphisms
$\chi:\cs(\A)\to \C$, then the Gelfand transform $\F$ is an isomorphism of
$\cs(\A)$ onto $C_{0}(\hA)$.  As usual, we write $\hat f= \F(f)$ for
all $f\in\cs(\A)$.  As shown in \cite{mrw:tams96}*{Corollary~3.4}, the
Gelfand dual $\hA$ is an abelian group bundle $\hp:\hA\to \go$ with
fibres $\hat A(u):=(A(u))^{\wedge}$.  If $\chi\in \hA$, the
corresponding complex homomorphism on $\cc(\A)$ is given by
\begin{equation}
  \label{eq:65}
  \chi(f) = \int_{A(\hp(\chi))} f(a)\overline{\chi(a)}
  \,d\beta^{\hp(\chi)}(a)=\hat f(\chi).
\end{equation}
(The complex conjugate appearing on the right-hand side of
\eqref{eq:65} is included to match up with our prejudice for the form
of the Fourier transform.)

Since $\A$ is abelian in this section, $\delta(\sigma)$ depends only
on $\dot\sigma$ and each $\beta^{u}$ is bi-invariant.  The right
action of $\Sigma$ on $\hA$ given by
\begin{equation}
  \label{eq:39}
  \chi\cdot \sigma(a):= \chi(\sigma a\sigma^{-1})
\end{equation}
factors through a right action of $\G$ on $\hA$.  So we may form the
transformation groupoid $\hA\rtimes \G$ for the action of $\G$ on the
space $\hA$.  Recall that we identify the unit space $\hA\rtimes\G$
with $\hA$.  Thus $r(\chi,\gamma)=\chi$,
$s(\chi,\sigma)=\chi\cdot \sigma$, and
$(\chi,\gamma)(\chi\cdot\gamma,\eta)=(\chi,\gamma\eta)$.  We can equip
$\hA\rtimes \G$ with the Haar system
$\halpha=\sset{\halpha^{\chi}}_{\chi\in\hA}=\sset{\delta_{\chi}\times
  \alpha^{\hp(\chi)}}$.

\subsection{The associated $\T$-groupoid}
\label{sec:assoc-t-group}

We want to build a $\T$-groupoid associated to $\Sigma$ just as in
\cite{mrw:tams96}*{\S4} except that there $\G$ was assumed to be
principal. We start by defining
\begin{equation}
  \label{eq:4aa}
  \D=\set{(\chi,z,\sigma)\in\hA\times\T\times\Sigma: \hp(\chi)=r(\sigma)}.
\end{equation}
We can make $\D$ into a locally compact Hausdorff groupoid by
identifying it with $\hA\rtimes\Sigma\times \T$.  Thus,
\begin{equation}
  \label{eq:5a}
  (\chi,z_{1},\sigma_{1})(\chi\cdot \sigma_{1},z_{2},\sigma_2) =
  (\chi,z_{1}z_{2},\sigma_{1}\sigma_{2}) \quad\text{and} \quad
  (\chi,z,\sigma)^{-1}= (\chi\cdot\sigma, \bar z, \sigma^{-1}).
\end{equation}
We can identify $\D^{(0)}$ with $\hA$, and then
\begin{equation}
  \label{eq:7a}
  r(\chi,z,\sigma)=\chi\quad\text{and} \quad
  s(\chi,z,\sigma) = \chi\cdot\sigma.
\end{equation}

Let $H$ be the subgroupoid of $\D$ consisting of triples of the form
$(\chi,\overline{\chi(a)},a)$ for $a\in A(\hp(\chi))$.  Note that if
$(\chi_{n},\overline{\chi_{n}(a_{n})} ,a_{n})\to (\chi,z,\sigma)$ in
$\D$, then $a_{n}\to \sigma$ and $\sigma=a\in A(\hp(\chi))$ since $\A$
is closed in $\Sigma$.  Then $\chi_{n}(a_{n})\to \chi(a)$ by
\cite{mrw:tams96}*{Proposition~3.3}.  Hence $H$ is closed in $\D$ with
$H^{(0)}=\D^{(0)}$.  To see that $H$ has an open range map, we use
Fell's Criterion (Lemma~\ref{lem-fell-criterion}).

So suppose that $\chi_{n}\to \chi=r(\chi,\overline{\chi(a)},a)$ in
$\hA\times\go$. Since $p:\A\to \go$ is open, we can pass to a subnet,
relabel, and assume that there are $a_{n}\to a$ in $\A$ such that
$p(a_{n})=\hp(\chi_{n})$.  Then
$(\chi_{n},\overline{\chi_{n}(a_{n})},a_{n}) \to
(\chi,\overline{\chi(a)},a)$ in $\D$.

We have now showed that $H$ has open range and source maps. Hence the
quotient map $q:\D\to \D/H$ is open. Furthermore, if $d\in \D$, then
it is not hard to see that $d H=Hd$. Thus, as in
Lemma~\ref{lem-normal-quotient}, we can form the locally compact
Hausdorff groupoid $\Sigmaw:=\D/H$, and the elements of $\Sigmaw$ are
given by triples $[\chi,z,\sigma]$ where
$[\chi,z,a\sigma]=[\chi,\chi(a) z ,\sigma]$ for all $a\in\A$. Thus we
can define maps $i:(\hA\times\T)\to\Sigmaw$ and
$j:\Sigmaw\to \hA\rtimes\G$ by
\begin{equation}
  \label{eq:11a}
  i(\chi,z)=[\chi,z,\hp(\chi)] \quad\text{and} \quad
  j([\chi,z,\sigma]) = (\chi,\dot\sigma).
\end{equation}

\begin{prop}
  \label{prop-t-groupoid} With respect to the maps $i$ and $j$ above,
  $\Sigmaw$ is a $\T$-groupoid over $\hA\rtimes\G$:
  \begin{equation}
    \label{eq:13a}
    \begin{tikzcd}[column sep=3cm]
      \hA\times\T \arrow[r,"i"] \arrow[dr,shift left, bend right = 15]
      \arrow[dr,shift right, bend right = 15]&\Sigmaw \arrow[r,"j",
      two heads] \arrow[d,shift left] \arrow[d,shift
      right]&\hA\rtimes\G \arrow[dl,shift left, bend left = 15]
      \arrow[dl,shift right, bend left = 15]
      \\
      &\hA&
    \end{tikzcd}
  \end{equation}
\end{prop}
\begin{proof}
  The map $i$ is clearly injective.  Suppose that
  $i(\chi_{n},z_{n})= [\chi_{n},z_{n},\hp(\chi_{n})]\to
  [\chi,z,\sigma]$.  Since $q:\D\to\Sigmaw$ is open, we may assume
  that there exist $a_{n}\to a$ in $\A$ such that
  $(\chi_{n},\overline{\chi_{n}(a_{n})}z_{n},a_{n}) \to
  (\chi,z,\sigma)$.  But then $\sigma=a$ and $z_{n}\to \chi(a)z$.
  Thus
  $[\chi,z,\sigma]=[\chi,z,a]=[\chi,\chi(a)z,\hp(\chi)]\in
  i(\hA\times\T)$ and $i$ has closed range.  Replacing $\sigma$ by
  $\hp(\chi)$ in the above shows that $z_{n}\to z$ and it follows that
  $i$ is a homeomorphism onto its range as required.

  If $j([\chi,z,\sigma])$ is a unit, then $\dot\sigma=\hp(\chi)$ and
  $\sigma=a\in A(\hp(\chi))$.  But then
  \begin{equation}
    \label{eq:12a}
    [\chi,z,a]=[\chi,\chi(a)z,\hp(\chi)]\in i(\hA\times\T).
  \end{equation}
  In other words, \eqref{eq:13a} is ``exact'' at $\Sigmaw$.

  We still need to check that $j$ is open.  Again, we employ Fell's
  criterion (see Lemma~\ref{lem-fell-criterion}).  Suppose that
  $(\chi_{n},\gamma_{n})\to (\chi,\dot\sigma)= [\chi,z,\sigma]$ in
  $\hA\rtimes \G$.  Since $p$ is open, we can pass to a subnet,
  relabel, and assume that there are $\sigma_{n}\to \sigma$ such that
  $p(\sigma_{n})=\gamma_{n}$.  But then
  $(\chi_{n},z,\sigma_{n})\to (\chi,z,\sigma)$.  This suffices.
\end{proof}

Given our Haar system on $\hA\rtimes\G$, we can build the restricted
\cs-algebra $\cshacgsw$.  Recall that this \cs-algebra is built from
functions $\tilde F\in \cc(\Sigmaw)$ such that
\begin{equation}
  \label{eq:22a}
  z'\tilde F([\chi,z,\sigma])=\tilde
  F([\chi,z'z,\sigma])\quad\text{for all $z'\in\T$.}
\end{equation}
Also recall that we write $\cc(\hacg;\Sigmaw)$ for the space of all
such functions. To make the notation easier to work with, notice that
any $\tilde F \in \cc(\hacg;\Sigmaw)$ is determined by its values on
classes of the form $[\chi,1,\sigma]$. Hence we identify
$\cc(\hacg;\Sigmaw)$ with the collection $\ccsas$ of continuous
functions $F$ on $\hacs$ such that
\begin{equation}
  \label{eq:23a}
  F(\chi,a\sigma) =\chi(a)F(\chi,\sigma) \quad\text{for all $a\in
    A(r(\sigma))$}
\end{equation}
and such that the support of $F$ has compact image in $\hacg$.

If $F,G \in\ccsas$, then
\begin{equation}
  \label{eq:18aa}
  F*G(\chi,\sigma)=\int_{\G} F(\chi,\tau) G(\chi\cdot
  \tau, \tau^{-1}\sigma) \,d\alpha^{r(\sigma)}(p(\tau))
\end{equation}
and
\begin{equation}
  \label{eq:19aa}
  F^{*}(\chi,\sigma) = \overline{F(\chi\cdot\sigma,\sigma^{-1})}.
\end{equation}

\subsection{The isomorphism}
\label{sec:isomorphism}

The Gelfand transform gives us an isomorphism of
$\cs(\A)=\Sec_{0}(\go;\E)$ onto $C_{0}(\hA)=\Sec_{0}(\go;\hat{\E})$
where $\hat\E=\coprod C_{0}(\hat A(u))$ is the bundle described just
before~\eqref{eq:65}.

Our constructions in the previous section give us a dynamical system
$(\hat\E,\Sigma,\hauto)$ where
$\hauto_{\sigma}:C_{0}(\hat A(s(\sigma)))\to C_{0}(\hat A(r(\sigma)))$
is given by
\begin{align}
  \label{eq:41}
  \hauto_{\sigma}(\hat h)(\chi)=\hat h(\chi\cdot \sigma).
\end{align}
The corresponding left-action of $\A$ on $r^{*}\hat\E$ is determined
on $\hat h\in C_{0}(\hat A(u))$ by
\begin{align}
  \label{eq:42}
  \hgtw(t)\hat h(\chi)=\overline{\chi(t)}\hat h(\chi).
\end{align}
We form the Fell bundle
$\hat\BB(\hat\E,\hat\auto,\hat\gtw) =\A\backslash r^{*}\hat\E$ for the
twist $\hat\gtw$ on $(\hat\E,\Sigma,\hat\auto)$. Sections
$\check g\in \Sec_{c}(\G;\hat\BB(\hat\E,\hat\auto,\hat\gtw))$ are
determined by $g\in\ccgshgtw$ where $g:\Sigma\to \hat\E$ is
continuous, and satisfies
\begin{align}
  \label{eq:43}
  g(t\sigma)(\chi)=\chi(t)g(\sigma)(\chi)
\end{align}
for all $(t,\sigma)\in\A*\Sigma$, and has support with compact image
in $\G$.

From Theorem~\ref{thm-hm}, we get isomorphisms
$\cs(\Sigma) \to \cs(\G;\hat\BB(\hat\E,\hat\auto,\hat\gtw))$ and
$\cs_r(\Sigma) \to \cs_r(\G;\hat\BB(\hat\E,\hat\auto,\hat\gtw))$ that
send $f\in \cc(\Sigma)$ to the section $\hat\renj(f)$ given by
$\hat\renj(\sigma) = {(\renj(f)(\sigma))\widehat{\;}}$. Hence
\begin{align}
  \label{eq:44aa}
  \hat\renj(f)(\sigma)(\chi)=\delta(\sigma)^{\half}\int_{\A}
  f(a\sigma)\overline{\chi(a)} \,d\beta^{r(\sigma)}(a).
\end{align}

Let $\cisas$ be the collection of all $f\in C_{0}(\hA*_{r}\Sigma)$
such that there is a compact set $K\subset \G$ such that
$f(\chi,\sigma)=0$ if $\dot\sigma\notin K$.  Since
$r^{*}\bigl(C_{0}(\hA)\bigr)\cong C_{0}(\Sigma*_{r}\hA)$, we obtain a
one-to-one correspondence between $\cisas$ and $\ccgshgtw$ that
carries $f\in\cisas$ to the element $F_{f}\in\ccgshgtw$ given by
\begin{equation}
  \label{eq:45aa}
  F_{f}(\sigma)(\chi)=f(\chi,\sigma).
\end{equation}

\begin{prop}
  \label{prop-easy-hm} The map $f\mapsto F_{f}$ is a $*$-isomorphism
  of $\ccsas$ into $\ccgshgtw$ which extends to isomorphisms
  \[
    \cshacgsw \cong \csgshgtw \quad\text{ and }\quad \csrhacgsw \cong
    \csrgshgtw.
  \]
\end{prop}
\begin{proof}
  We first prove the isomorphism of full $C^*$-algebras. If
  $\chi\in\hat A(r(\sigma))$, then since evaluation at $\chi$ passes
  through the integral,
  \begin{align}
    \label{eq:50}
    F_{f}*F_{g}(\sigma)(\chi)
    &= \int_{\G} F_{f}(\tau)(\chi)
      \hauto_{\tau}(F_{g}(\tau^{-1}\sigma)(\chi)
      \,d\alpha^{r(\sigma)}(\dot \tau)  \\
    &= \int_{\G} F_{f}(\tau)(\chi)
      F_{g}(\tau^{-1}\sigma)(\chi\cdot \tau)
      \,d\alpha^{r(\sigma)}(\dot\tau) \\
    &= \int_{\G}f(\chi,\tau) g(\chi\cdot\tau,\tau^{-1}\sigma)
      \,d\alpha^{r(\sigma)} (\dot\tau).
  \end{align}
  Hence $F_{f*g}=F_{f}*F_{g}$.  A similar computation shows that
  $F_{f^{*}}= F_{f}^{*}$.  Thus $f\mapsto F_{f}$ is a $*$-isomorphism
  onto its range.

  The $\|\cdot\|_{I}$-norm on $\ccgshgtw$ is given by
  \begin{equation}
    \label{eq:46}
    \|F_{f}\|_{I}=\max\big\{
    \|F_{f}\|_{I,r},\|F_{f}\|_{I,s}\bigr\},
  \end{equation}
  where
  \begin{align}
    \label{eq:47}
    \|F_{f}\|_{I,r}=
    \sup_{u\in\go}\int_{\G}\|F_{f}(\sigma)\|_{\infty}
    \,d\alpha^{u}(\dot\sigma)
    \quad\text{and}\quad
    \|F_{f}\|_{I,s}=\sup_{u\in\go}\int_{\G}
    \|F_{f}(\sigma)\|_{\infty}\,d\alpha_{u}(\dot \sigma).
  \end{align}

  The set $\set{F_{f}:f\in\ccsas}$ is clearly dense in $\ccgshgtw$ in
  this $\|\cdot\|_{I}$-norm.

  There exists $\chi_{\sigma}\in \hA(r(\sigma))$ such that
  $\|F_{f}(\sigma)\|_{\infty}=|f(\chi_{\sigma},\sigma)|$.  Thus
  \begin{align}
    \label{eq:48}
    \|F_{f}\|_{I,r}
    &=\sup_{u\in\go}\int_{\G}|f(\chi_{\sigma},\sigma)|
      \,d\alpha^{u}( \dot\sigma)  \\
    &= \sup_{\substack{u\in\go \\ \hp(\chi)=u}} \int_{\G} |f(\chi,\sigma)|
    \,d\alpha^{u} (\dot\sigma) \\
    &=\sup_{\chi\in\hA} \int_{\G}|f(\chi,\sigma)|
      \,d\alpha^{\hp(\chi)}(\dot \sigma) \\
    &=\|f\|_{I,r}.
  \end{align}
  Similarly, $\|F_{f}\|_{I,s}=\|f\|_{I,s}$, and $f\mapsto F_{f}$ is
  isometric for the respective $I$-norms. The isomorphism of full
  $C^*$-algebras follows.

  For the isomorphism of reduced $C^*$-algebras, let $\H(\hat{\BB})$
  be the right-Hilbert $C_0(\hA)$-module described in
  Section~\ref{sec:Fell-bundles-csa} for the Fell bundle $\hat{\BB}$.
  Regard $\H(\hat{\BB})$ as a right $C_0(\hA)$-module by identifying
  $C_0(\G^{(0)}; \hat{\E})$ with $C_0(\hA)$ via
  $f \mapsto \bigl(\chi \mapsto f(p(\chi))(\chi)\bigr)$. Then
  $C_c(\G, \hat{\BB})$ acts on the left of $\H(\hat{\BB})$ by
  convolution, and the map implementing this action is isometric for
  the reduced norm on $C_c(\G, \hat{\BB})$. Now consider the Hilbert
  module $\H(\hA \rtimes \G; \hat{\Sigma})_{C_0(\hA)}$ obtained from
  the twist $\hat{\Sigma}$ as described in
  Section~\ref{sec:Fell-bundles-csa}, so that the left action of
  $C_c(\hA \rtimes \G; \hat{\Sigma})$ on
  $\H(\hA \rtimes \G; \hat{\Sigma})$ by convolution is isometric for
  the reduced norm on $C_c(\hA \rtimes \G; \hat{\Sigma})$. A
  straightforward calculation shows that
  $\langle f, g\rangle_{C_0(\hA)} = \langle F_f,
  F_g\rangle_{C_0(\hA)}$ and so $f \mapsto F_f$ extends to an
  isomorphism of Hilbert modules
  $\H(\hA \rtimes \G; \hat{\Sigma}) \cong \H(\hat{\BB})$, which
  intertwines the left actions because $f \mapsto F_f$ is a
  homomorphism. Hence
  \[
    \|f\|_{\csrhacgsw} = \|f\|_{\Ll(\H(\hA \rtimes \G; \hat{\Sigma}))}
    = \|F_f\|_{\Ll(\H(\hat{\B}))} = \|F_f\|_{\csrgshgtw}.\qedhere
  \]
\end{proof}

Since we can identify $\cisas$ and $\ccgshgtw$, we can view $\cisas$
as a dense subalgebra of $\cshacgsw$.  Furthermore, if
$f\in\cc(\Sigma)$, then
\begin{equation}
  \label{eq:49}
  \Phi(f)(\chi,\sigma)=\delta(\sigma)^{\half}\int_{\A} f(a\sigma)
  \overline{\chi(a)}\,d\beta^{r(\sigma)}(a)
\end{equation}
defines an element of $\cisas$.  Noticing that
$\Phi(f)(\chi,\sigma)=\renj(f)(\sigma)(\chi)$, we can combine
Theorem~\ref{thm-hm} and Proposition~\ref{prop-easy-hm} to obtain the
main result in this section.

\begin{thm}
  \label{thm-main-mrw96} Let $\Sigma$ be a second countable locally
  compact groupoid with a Haar system $\lambda$.  Suppose that $\A$ is
  a closed abelian normal subgroup bundle with a Haar system $\beta$.
  Then the map $\Phi:\cc(\Sigma)\to \cisas$ given in \eqref{eq:49}
  extends to an isomorphism of $\cs(\Sigma)$ onto the restricted
  \cs-algebra $\cshacgsw$. Moreover, this isomorphism descends to an
  isomorphism $\cs_r(\Sigma) \cong \cs_r(\hA\rtimes\G;\Sigmaw)$.
\end{thm}

\begin{remark} \label{rem-4.2} Theorem~\ref{thm-main-mrw96}
  generalizes \cite{mrw:tams96}*{Proposition~4.5} where it was assumed
  that $\A=\Sigma'$, so that $\G$ is principal, and that $\cs(\Sigma)$
  is CCR.
\end{remark}

\section{Examples and applications}\label{sec:more examples}

\subsection{Closed normal abelian subgroups}\label{sec:clos-norm-abel}

Let $G$ be a locally compact group, and let $H \le G$ be a closed
normal abelian subgroup. Putting $\mathcal{A} := H$, $\Sigma := G$ and
$\G := G/H$, we obtain an instance of the situation of
Section~\ref{sec:abelian-case}, in which the groupoids involved have a
single unit.

The group $G$ acts on the left of $H$ by conjugation, this descends to
an action of $G/H$ because $H$ is abelian, and these left actions
induce right actions of $G$ and $G/H$ on the space $\widehat{H}$. So
we obtain an extension of groupoids
\[
  \begin{tikzcd}
    \widehat{H} \times H\arrow[r]& \widehat{H} \rtimes G \arrow[r] &
    \widehat{H} \rtimes (G/H)
  \end{tikzcd}
\]
with common unit space $\widehat{H}$ (regarded as a topological
space).

The groupoid $\D$ of Section~\ref{sec:assoc-t-group} is the cartesian
product $(\widehat{H} \rtimes G) \times \T$ of the transformation
groupoid for the action of $G$ on $\widehat{H}$, with the circle
group. The closed subgroupoid $\iota(\widehat{H} \times H)$ is the set
$\{((\chi, h), \overline{\chi(h)}) : (\chi,h) \in \widehat{H} \times
H\}$. So the associated $\T$-groupoid of
Section~\ref{sec:assoc-t-group} is the quotient
\[
  \widehat{\Sigma} := \{[(\chi,g), z] : (\chi,g) \in \widehat{H}
  \rtimes G, z \in \T\},
\]
in which $[(\chi,g), z] = [(\chi, hg), \overline{\chi(h)}z]$ for any
$h \in H$.

The inclusion
$\iota : (\chi, h) \mapsto [(\chi, h), \overline{\chi(h)}]$ and the
groupoid homomorphism $\pi : [(\chi, g), z] \mapsto (\chi, gH)$ yield
the twist
\begin{equation}
  \label{eq:83}
  \begin{tikzcd}
    \widehat{H} \times \T \arrow[r,"\iota"] & \widehat{\Sigma}
    \arrow[r,"\pi"] & \widehat{H} \rtimes (G/H),
  \end{tikzcd}
\end{equation}
and Theorem~\ref{thm-main-mrw96} yields an isomorphism
$\cs(G) \cong \cs(\widehat{H} \rtimes (G/H); \widehat{\Sigma})$.

\begin{remark}\label{rmk:clopen subgroup}
  An interesting special case of the construction described above
  occurs when the closed normal abelian subgroup $H$ is a clopen
  subgroup as studied by Zeller-Meier \cite{zel:jmpa68}*{2.31}.  This
  implies first that the quotient group $G/H$ is a discrete group, and
  second that the group $C^*$-algebra $C^*(H) \cong C_0(\widehat{H})$
  is a subalgebra of $C^*(G)$ that contains an approximate identity
  for $C^*(G)$. Since $G/H$ is discrete, there exists a continuous
  section $\widehat{H} \rtimes (G/H) \to G$ for the quotient
  map. Therefore the twist $\widehat{\Sigma}$ is topologically
  trivial, in the sense that it is determined by a continuous
  $\T$-valued $2$-cocycle on $\widehat{H} \rtimes (G/H)$.
\end{remark}

\begin{example}\label{ex:Heisenberg}
  As an illustrative example, consider the integer Heisenberg group
  \begin{equation}
    \label{eq:87}
    G = \langle \,\text{$a$, $b$, $c$} \mid
    \text{$ab = c ba$, $ca = ac$, $cb = bc$}\,\rangle.
  \end{equation}
  (A similar analysis applies for the Heisenberg group of
  upper-triangular $3 \times 3$ real matrices with diagonal entries
  equal to 1, but the integer example is slightly easier to describe.)
  We consider two natural (clopen) normal abelian subgroups of $G$ in
  the context of our result.
  \begin{compactenum}
  \item \label{it:Heis1} First consider the subgroup $H_c \cong \Z$
    generated by the element $c$.  This $H_c$ is precisely the center
    $\mathcal{Z}(G)$, and so the action of $G/H_c \cong \Z^2$ on
    $\widehat{H}_c \cong \T$ is trivial. Hence the semidirect product
    $\widehat{H}_c \rtimes (G/H_c)$ is just the product
    $\widehat{H}_c \times (G/H_c) \cong \T \times \Z^2$ regarded as
    group bundle. As discussed in Remark~\ref{rmk:clopen subgroup},
    $C^*(H_c) \cong C(\T)$ embeds as a unital subalgebra of $C^*(G)$
    which is central because $H_c$ is central, and therefore makes
    $C^*(G)$ into a $C(\T)$-algebra. For $z \in \T$, the fibre of
    $\widehat{\Sigma}$ corresponding to the character
    $\chi_z : c \mapsto z$ of $H_c$ is the extension of $\Z^2$
    corresponding to the $2$-cocycle
    $\bigl((m_{1},m_{2}),(n_{1},n_{2})\bigr) \mapsto z^{m_2 n_1}$. The
    corresponding fibre of $C^*(H)$ is isomorphic to the rotation
    algebra $A_\theta$ where $z = e^{2\pi i \theta}$. So we recover,
    in this instance, Anderson and Paschke's description
    \cite{andpas:hjm89} of $C^*(G)$ as the section algebra of a field
    of rotation algebras.
  \item Now consider the subgroup $H_{b,c} \cong \Z^2$ of $G$
    generated by $b$ and $c$, so that $\widehat{H}_{b,c} \cong
    \T^2$. The quotient $G/H_{b,c}$ is isomorphic to $\Z$ via
    $aH_{b,c} \mapsto 1$, and acts on $\T^2$ by
    $1 \cdot (w,z) = (zw, z)$, inducing an action $\alpha$ of $\Z$ on
    $C(\T^2)$. 
    The twist $\widehat{\Sigma}$ is the trivial twist over
    $\T^2 \rtimes \Z$, and so we recover the well-known description
    $C^*(G) \cong C^*(\T^2 \rtimes \Z;\widehat{\Sigma} ) \cong
    C^*(\T^2 \rtimes \Z)$ as the crossed-product algebra
    $C(\T^2) \rtimes_\alpha \Z$.
  \end{compactenum}
\end{example}

\begin{remark}
  Williams proved in \cite{wil:tams81}*{pp 357--358} that when $G/H$
  is abelian and $G$ is a semidirect product,
  $G \cong H \rtimes (G/H)$, then
  $C^*(G) \cong C^*(\widehat{H} \rtimes (G/H))$. Note that the
  extension $\widehat{\Sigma}$ is trivial.
\end{remark}

\begin{remark}
  The situation described in
  Example~\ref{ex:Heisenberg}(\ref{it:Heis1}) generalises as
  follows. If $H$ is any closed subgroup of $\mathcal{Z}(G)$, then
  $C^*(H)$ embeds in the centre of the multiplier algebra of $C^*(G)$,
  making $C^*(G)$ into a $C_0(\widehat{H})$-algebra. The
  transformation group $\widehat{H} \rtimes (G/H)$ coincides with the
  group bundle $\widehat{H} \times (G/H)$. Each $\chi \in \widehat{H}$
  determines a central extension
  $\T \longrightarrow \widehat{\Sigma}_\chi \longrightarrow G/H$ in
  which $\widehat{\Sigma}_\chi$ coincides with the quotient group
  $(G \times \T)/\{(h, \overline{\chi(h)}) : h \in H\}$. The fibre of
  $C^*(G)$ corresponding to a character $\chi$ of $H$ is the twisted
  group $C^*$-algebra of this extension.
\end{remark}

\subsection{Extensions of effective \'etale groupoids}\label{sec:etale
  effective}
For the convenience of the reader we recall our standing hypotheses
and the extension~\eqref{eq:ext} where $\Sigma$ is a unit space fixing
extension of a subgroupoid group bundle $\A$.  That is, we have
\begin{equation}
  \label{eq:ext2}
  \begin{tikzcd}[column sep=3cm]
    \A \arrow[r,"\iota", hook] \arrow[dr,shift left, bend right = 15]
    \arrow[dr,shift right, bend right = 15]&\Sigma \arrow[r,"p", two
    heads] \arrow[d,shift left] \arrow[d,shift right]&\G
    \arrow[dl,shift left, bend left = 15] \arrow[dl,shift right, bend
    left = 15]
    \\
    &\go&
  \end{tikzcd}
\end{equation}
where $\iota$ is the inclusion map and $p$ is a continuous open
surjection restricting to a homeomorphism of $\go$ and $\goo$.  We
require that all groupoids in the extension are second countable
locally compact, Hausdorff, both $\Sigma$ and $\A$ have Haar systems
and $\A$ is an abelian group bundle.

In this section,
we consider the situation where the groupoid $\G$ in the
extension~\eqref{eq:ext2} is an effective \'etale groupoid; that is,
$r, s : \G \to \G^{(0)}$ are local homeomorphisms, and the topological
interior of the isotropy bundle of $\G$ is $\G^{(0)}$.  Since our
standing hypotheses include that $\G$ is second-countable and
Hausdorff, the latter is equivalent to the condition that $\G$ is
topologically principal in the sense that the set
$\{x \in \G^{(0)} : \G^x_x = \{x\}\}$ is dense in $\G^{(0)}$
\cite{ren:irms08}*{Proposition~3.6}.

It follows that the transformation groupoid $\hA \rtimes \G$ is also
\'etale and effective; this is certainly well known but also easy to
prove directly:

\begin{lemma}
  Let $\G$ be a locally compact, Hausdorff groupoid acting on the
  right of a locally compact Hausdorff space $X$. If $\G$ is \'etale
  (resp., effective) then so is $X \rtimes \G$.
\end{lemma}
\begin{proof}
  Suppose that $\G$ is \'etale and fix $(x, \gamma) \in X \rtimes
  \G$. Fix an open bisection neighbourhood $U$ of $\gamma$ in $\G$ and
  an open neighbourhood $W$ of $x \in X$. Then
  $(W \times U) \cap (X \rtimes \G) = \{(y, \eta) \in W \times U \mid
  s(y) = r(\gamma)\}$ is an open bisection neighbourhood of
  $(x, \gamma)$.  Hence $X \rtimes \G$ is \'etale.

  Now suppose that $\G$ is effective. Suppose that
  $U \subseteq X \rtimes \G$ is open and consists entirely of
  isotropy. Then it is a union of sets of the form $W * V$ where
  $W \subseteq X$ is open, and $V \subseteq \G$ is an open bisection
  consisting of isotropy. For any such set, since $\G$ is effective,
  we have $V \subseteq \G^{(0)}$, and so
  $W * V \subseteq X * \G^{(0)} = (X \rtimes \G)^{(0)}$.
\end{proof}

It follows from our main theorem that $C^*_r(\Sigma)$ is equal to the
reduced $C^*$-algebra of a twist $\widehat{\Sigma}$ over the \'etale
effective groupoid $\hA \rtimes \G$, and so Renault's theory of Cartan
subalgebras of $C^*$-algebras applies. In the case where $\Sigma$ is
\'{e}tale the first assertion in the following theorem may be deduced
from \cite{bnrsw:ieot16}*{Corollary~4.5}. In addition, we provide an
explicit construction of the Weyl twist which is not included in
\cite{bnrsw:ieot16}. Our result is also related to Theorem~5.8 of  
\cite{dgnrw:jfa20}: in their case $\Sigma$ is \'{e}tale but 
they also consider a continuous $\T$-valued $2$-cocycle.

\begin{thm}\label{cartan}
  Suppose that $\Sigma$ is an extension of a wide normal abelian
  subgroup bundle with a Haar system as in~\eqref{eq:ext2}.  Further
  assume that $\G$ is \'etale and effective. Then
  $C^*(\mathcal{A}) \cong C_0(\hA)$ is a Cartan subalgebra of
  $C^*_r(\Sigma)$. The Weyl twist of the Cartan pair
  $(C^*_r(\Sigma), C^*(\A))$ is $(\hA \rtimes \G, \widehat{\Sigma})$.
\end{thm}
\begin{proof}
  We apply \cite{ren:irms08}*{Theorem~5.2} to the twist
  $\hA\times\T \longrightarrow \Sigmaw \longrightarrow \hA\rtimes \G$
  to see that the pair
  $(C^*_r(\hA \rtimes \G; \widehat{\Sigma}), C_0(\hA))$ is a Cartan
  pair with Weyl twist $(\hA \rtimes \G; \widehat{\Sigma})$. The
  result then follows from the final statement of
  Theorem~\ref{thm-main-mrw96}.
\end{proof}

Many of our examples do not require the full power of our
theorem. Note, though, that in the next example $\Sigma$ is not
\'{e}tale and so \cite{bnrsw:ieot16}*{Corollary~4.5} cannot be used.

\begin{example}
  Let $(A, B)$ be a Cartan pair and let $(\G, \Sigma)$ be the
  associated Weyl twist. Then $\Sigma$ is an extension of $\G$ by the
  wide normal abelian subgroup bundle $\A \cong \T \times \G^{(0)}$.
  Moreover, since $(A, B)$ is a Cartan pair, $\G$ is \'etale and
  effective. Hence, by Theorem~\ref{cartan},
  $C^*(\A) \cong C_0(\Z \times \G^{(0)})$ is a Cartan subalgebra of
  $C^*_r(\Sigma)$.
\end{example}

\begin{example}
  Let $X$ be a locally compact Hausdorff space and let
  $n \mapsto \sigma^{n}$ be an action of $\NN^k$ by local
  homeomorphisms of $X$. Let $\Sigma$ be the associated
  Deaconu--Renault groupoid
  $\{(x, m-n, y) : \sigma^m(x) = \sigma^n(y)\}$. Let $\A$ denote the
  interior of the isotropy in $\Sigma$, and suppose that $\A$ is a
  clopen subset of $\Sigma$ (this may seem like a restrictive
  hypothesis but it is automatic if $\Sigma$ is minimal by
  \cite{kps:jncg16}*{Proposition~2.1} applied to the cocycle
  $c(x, m, y) = m$, in which case it is given by
  $\A = \{(x,m,x) : (y,m,y) \in \Sigma\text{ for all }y \in X\}$
  \cite{simwil:art15}*{Proposition~3.10}). Then $\mathcal{A}$ is an
  abelian group bundle. Specifically, for each $x \in X$, the set
  \[
    \Stab^{\operatorname{ess}}_{x}\{p - q : p,q \in \Z^k\text{ and
    }\sigma^p = \sigma^q\text{ on a neighbourhood of }x\}
  \]
  is a subgroup of $\Z^k$, we have
  $\Stab^{\operatorname{ess}}_x = \Stab^{\operatorname{ess}}_y$
  whenever $\Sigma^x_y \not= \emptyset$, and as a set,
  \[
    \A = \bigsqcup_{x \in X} \{x\} \times
    \Stab^{\operatorname{ess}}_x.
  \]
  Since $\A$ is clopen, we can form the quotient groupoid
  $\G = \Sigma/\A$, which is \'etale and effective. A quick
  calculation shows that the action of $\G$ on $\A$ is given by
  $\gamma \cdot (s(\gamma), m) = (r(\gamma), m)$, and so the induced
  action of $\G$ on $\hA$ is given by
  $(r(\gamma), \chi) \cdot \gamma = (s(\gamma), \chi)$ for
  $\chi \in (\Stab_{r(\gamma)})\widehat{\;}$. So, as a groupoid,
  $\hA \rtimes \G$ is just the fibred product $\hA * \G$---for example
  if $\G$ is minimal, then $\hA \rtimes \G \cong \hA_x \times \G$ for
  any $x \in X$. Since $\A$ is open in $\Sigma$, the map $p_{\A}$ is
  open and Theorem~\ref{thm-main-mrw96} applies.  Combining with
  Theorem~\ref{cartan}, this implies that
  $(C^*(\hA \rtimes \G; \widehat{\Sigma}), C_0(\hA))$ is a Cartan pair
  with Weyl twist $(\hA \rtimes \G; \widehat{\Sigma})$.
\end{example}

\begin{example}
  As a particular case of the preceding example, let $\Lambda$ be a
  row-finite $k$-graph with no sources, and let $\G_\Lambda$ be the
  associated infinite-path groupoid, which is the Deaconu--Renault
  groupoid for the action of $\NN^k$ on $\Lambda^\infty$ by shift maps
  \cite{kumpas:nyjm00}. Suppose that the interior $\mathcal{I}$ of the
  isotropy in $\G_\Lambda$ is closed---for example, this is automatic
  if $\Lambda$ is cofinal, in which case
  $\mathcal{I} \cong \Lambda^\infty \times
  \operatorname{Per}(\Lambda)$ \cite{kps:jncg16}*{Corollary~2.2}. (See
  \cite{bly:mjm17}*{Theorem~4.4} for a characterisation of when
  $\mathcal{I}$ is closed in $\G_\Lambda$.) Then our main theorem
  implies that $C^*(\mathcal{I}) \cong C_0(\hat{\mathcal{I}})$ is a
  Cartan subalgebra of $C^*(\Lambda)$, recovering
  \mbox{(c)\;$\implies$\;(b)} of \cite{bnrsw:ieot16}*{Corollary~4.6},
  and also gives a recipe for describing the associated Weyl twist.
\end{example}

\subsection{Transformation groupoids}
\label{sec:tr_gr}

Assume that $\Gamma$ is a locally compact Hausdorff groupoid with Haar
system which acts on the right on a locally compact space $X$. We
write $s_{X}:X\to \Gamma^{(0)}$ for the moment map and we let
$\Sigma:=X\rtimes \Gamma$ be the transformation groupoid. Recall that
$X\rtimes \Gamma=\set{ (x,\gamma)\in X\times
  \Gamma:s_{X}(x)=r(\gamma)}$ with the topology inherited from the
product topology and the range and source maps given by
$r(x,\gamma)=(x,r(\gamma))$ and
$s(x,\gamma)=(x\cdot \gamma,s(\gamma))$.  Then
$(x,\gamma)(x\cdot \gamma,\eta)=(x,\gamma\eta)$ while
$(x,\gamma)^{-1}=(x\cdot \gamma,\gamma^{-1})$.  We identify the unit
space of $X\rtimes \Gamma$ with $X$ as usual via the map
$(x,r(\gamma))\to x$.  Assume that $\N$ is a wide normal subgroup
bundle of $\Gamma$ which acts trivially on $X$: $x\cdot a=x$ for all
$x\in X$ and $a\in \N$. Then the action of $\Gamma$ on $X$ descends to
an action of $\Gamma/\N$ on $X$. Therefore, if we let $\A:=X\ast \N$
be the pull-back bundle and let $\G:=X\rtimes \Gamma/\A$, we obtain a
groupoid extension as in \eqref{eq:ext} with $\go=X$. If $\N$ is
abelian, then $\A=X*\N$ is also abelian. Then the $\T$-groupoid
  $\Sigmaw$ constructed in Section~\ref{sec:assoc-t-group} can be
  viewed as equivalence classes with $x\in X$,
  $\chi\in \hat{N}_{s(x)}$, $\gamma\in \Gamma$ and $z\in \T$ such that
  $\bigl[(\chi,z,x,a\gamma)\bigr] =\bigl[(\chi,\chi(a)z,x,\gamma)
  \bigr]$.  The transformation groupoid $\hA\rtimes\G$ is represented
  by triples $\set{(\chi,x,\dot\gamma): \text{$\chi\in \hat N_{s(x)}$
      and $r(\gamma)=s(x)$}}$.  Then Theorem~\ref{thm-main-mrw96}
  implies that the transformation groupoid \cs-algebra $\cs(\Sigma)$
  is isomorphic to the restricted groupoid \cs-algebra
  $\cs(\hA\rtimes\G;\Sigmaw)$ and similarly for the reduced algebras.


\begin{example}[Rational Rotation Algebra]
  We consider the group $\Gamma=\Z$ acting on $\T$ via rotation by
  $\alpha\in \Q$: $z\cdot k:=ze^{2\pi ik\alpha}$. If $\alpha=m/n$ with
  $m$ and $n$ relatively prime, then $n\Z$ fixes the action.  Hence we
  get an instance of the above situation with $\N$ the group $n\Z$,
  $\Gamma=\Z$ and $\G=\Z_n$. Therefore the groupoid extension is
  \begin{equation}
    \label{eq:rational}
    \begin{tikzcd}
      \T\times n\Z \arrow[r,"i",hook] & \T\rtimes \Z \arrow[r,"p",two
      heads] & \T \rtimes \Z_n.
    \end{tikzcd}
  \end{equation}
  The $C^*$-algebra $C^*(\T\rtimes \Z)$ is the rational rotation
  $C^*$-algebra $\A_\alpha$ (see, for example, \cite{DeBr:84}). If we
  let $\T_{n}=\T/\Z_{n}$ be the dual of $n\Z$, then $\Sigmaw$ consists
  of equivalence classes of elements of $\T_{n}\times\T^{2}\times \Z$
  where $[\chi,w,z,nl+k]=[\chi^{nl},w,z,k]$.
  In particular, we realize the rotation algebra $\A_{\alpha}$ as the
  appropriate completion of continuous functions $F$ on
  $\T_{n}\times\T\times\Z$ such that
  $F(\chi,w,nl+k)=\chi^{nl}F(\chi,w,k)$.
\end{example}

\begin{remark}\label{rem-marius-plus}
  The arguments from the above example can be easily extended to
  actions of $\Z$ on compact spaces by periodic homeomorphisms.  These
  ideas can be sharpened by viewing the construction of $\Sigmaw$ in
  Section~\ref{sec:assoc-t-group} as an instance of general
  ``pushout'' construction that will purse in a future project
  \cite{ikrsw:xx20b}.
\end{remark}

\begin{example}
  A large class of examples is given by twisted groupoid crossed products
  $(\G,\Sigma,\E,\auto,\gtw)$ as defined in Section
  \ref{sec:group-cross-prod} under the assumption that the \cs-bundle
  $\E$ consists of abelian \cs-algebras.  Then the $C_0(\go)$-algebra
  $A=\Sec_0(\go;\E)$ is an abelian \cs-algebra and, thus, isomorphic
  to $C_0(X)$, where $X$ is the Gelfand spectrum of $A$. Since $A$ is
  a $C_0(\go)$-algebra the bundle map $s:X\to \go$ is
  continuous. Moreover, $X=\bigsqcup X_u$, where $X_u$ is the spectrum
  of the fibre $E(u)$. The action of $\Sigma$ on $\E$ induces a right
  action on $X$ via $\auto(f)(x)=f(x\cdot \sigma)$ for all
  $f\in C_0(X_{s(\sigma)})$. Since the action of $\A$ on $\E$ is
  unitarily implemented as in \eqref{eq:61} and since $E(u)$ is
  abelian, $\A$ acts trivially on $X$. Thus we obtain an extension as
  above with $\N:=\A$ and $\Gamma:=\Sigma$.  Moreover, $\csgsgtw$ is
  isomorphic to $C^*(X\rtimes \Sigma)$. Indeed one can prove that the
  map that sends $f\in C_c(X\rtimes \Sigma)$ to $\hat{f}\in \ccgsgtw$
  defined via
  \[
    \hat{f}(\sigma)(x):=\delta(\sigma)^{1/2}\int_{\A}
    f(x,a\sigma)\overline{\gtw(a)}\,d\beta^{r(\sigma)}(a)
  \]
  extends to an isomorphism of $C^*$-algebras. In the above equation
  $\{\beta^u\}$ is the Haar system on
  $\A$ 
  and $\delta$ is the modular map defined in Lemma
  \ref{lem-omega}. Theorem \ref{thm-main-mrw96} implies that
  $\csgsgtw$ is isomorphic to the restricted $C^*$-algebra
  $\cshacgsw $, where $\Sigmaw$ and $\hA\rtimes \G$ are described
  above. Moreover, the isomorphism descends to the level of reduced
  $C^*$-algebras.
\end{example}



\def\sp{^}

\end{document}